\documentclass[11pt,a4paper]{article}
\usepackage{geometry}

\usepackage[
    backend=biber,
    backref=false,
    url=false,
    doi=false,
    isbn=false,
    eprint=true
    ]{biblatex}

\renewbibmacro{in:}{}  
\AtEveryBibitem{\clearfield{pages}}  

\DeclareSourcemap{
  \maps[datatype=bibtex]{
    \map{
      \step[fieldsource=doi,final]
      \step[fieldset=url,null]
    }  
  }
}

\usepackage{xurl}

\usepackage{hyperref}[
    bookmarks=true,
    linktocpage=true,
    bookmarksnumbered=true,
    breaklinks=true,
    pdfstartview=FitH,
    hyperfigures=false,
    plainpages=false,
    naturalnames=true,
    colorlinks=true,
    pagebackref=false,
    pdfpagelabels
]

\hypersetup{
	colorlinks,
	citecolor=blue,
	filecolor=blue,
	linkcolor=blue,
	urlcolor=blue
}

\usepackage{amsmath, amsthm, amssymb}
\usepackage{thm-restate}
\usepackage[capitalize,noabbrev]{cleveref}
\usepackage{tikz-cd}

\usepackage[marginpar]{todo}
\usepackage{mathtools}

\usepackage{algorithmic}
\usepackage[section]{algorithm}

\theoremstyle{plain}
\newtheorem{thm}{Theorem}[section]
\Crefname{thm}{Theorem}{Theorems}
\newtheorem{cor}[thm]{Corollary}
\newtheorem{lem}[thm]{Lemma}
\newtheorem{prop}[thm]{Proposition}
\Crefname{prop}{Proposition}{Propositions}

\theoremstyle{definition}
\newtheorem{defi}[thm]{Definition}
\newtheorem{ex}[thm]{Example}
\newtheorem{rem}[thm]{Remark}

\definecolor{myblue}{rgb}{0.22,0.45,0.70}
\definecolor{myred}{RGB}{228,36,20}
\definecolor{mygreen}{RGB}{40,168,55}
\definecolor{myyellow}{RGB}{255,204,4}
\definecolor{mypurple}{RGB}{160,80,154}
\definecolor{mygray}{gray}{0.85}
\definecolor{mylightgray}{gray}{0.9}

\DeclareMathOperator{\coker}{coker}
\DeclareMathOperator{\im}{im}
\DeclareMathOperator{\coim}{coim}

\DeclareMathOperator{\id}{id}

\DeclareMathOperator{\colim}{colim}

\DeclareMathOperator{\Hom}{Hom}
\DeclareMathOperator{\supp}{supp}

\newcommand{\deathPart}[1]{#1^{\dagger}}
\newcommand{\birthPart}[1]{#1^{*}}
\newcommand{\ancientPart}[1]{#1^{-\infty}}
\newcommand{\ancientDeathPart}[1]{#1^{-\infty,\dagger}}
\newcommand{\immortalBirthPart}[1]{#1^{*,\infty}}
\newcommand{\immortalPart}[1]{#1^{\infty}}
\newcommand{\finitePart}[1]{#1^{\dagger,*}}
\newcommand{\constantPart}[1]{#1^{-\infty,\infty}}
\newcommand{\unbornPart}[1]{#1^{\triangleleft}}
\newcommand{\ghostlikePart}[1]{#1^{\triangleright}}
\newcommand{\lifespan}[1]{#1^{\diamond}}
\newcommand{\deathDualPart}[1]{#1_{\dagger}}
\newcommand{\birthDualPart}[1]{#1_{*}}
\newcommand{\ancientDualPart}[1]{#1_{-\infty}}
\newcommand{\ancientDeathDualPart}[1]{#1_{-\infty,\dagger}}
\newcommand{\immortalBirthDualPart}[1]{#1_{*,\infty}}
\newcommand{\immortalDualPart}[1]{#1_{\infty}}
\newcommand{\finiteDualPart}[1]{#1_{\dagger,*}}
\newcommand{\constantDualPart}[1]{#1_{-\infty,\infty}}
\newcommand{\unbornDualPart}[1]{#1_{\triangleleft}}
\newcommand{\ghostlikeDualPart}[1]{#1_{\triangleright}}
\newcommand{\dualLifespan}[1]{#1_{\diamond}}

\newcommand{\colimUnit}{\eta}
\newcommand{\limCounit}{\epsilon}
\newcommand{\dual}[1]{#1^{\vee}}

\title{Lifespan Functors and Natural Dualities in~Persistent~Homology}

\author{Ulrich Bauer \and Maximilian Schmahl}

\begin{document}

\maketitle

\begin{abstract}
We introduce lifespan functors, which are endofunctors on the category of persistence modules that filter out intervals from barcodes according to their boundedness properties. They can be used to classify injective and projective objects in the category of barcodes and the category of pointwise finite-dimensional persistence modules. They also naturally appear in duality results for absolute and relative versions of persistent (co)homology, generalizing previous results in terms of barcodes. Due to their functoriality, we can apply these results to morphisms in persistent homology that are induced by morphisms between filtrations. This lays the groundwork for the efficient computation of barcodes for images, kernels, and cokernels of such morphisms.
\end{abstract}

\section{Introduction}
\emph{Persistent homology}, the homology of a filtration of simplicial complexes, is a cornerstone in the foundations of topological data analysis. The most common setting studied in the literature is as follows: Given a filtration of simplicial complexes 
\[
K_{\bullet}\colon \emptyset=K_{-\infty}=\dots=K_0\subseteq K_1\subseteq\dots\subseteq K_{N}=\dots=K_{\infty}=K,
\]
by applying homology with coefficients in a field to each space and to each inclusion map
one obtains a diagram of vector spaces 
\[
H_{*}(K_{\bullet})\colon \dots\to H_{*}(K_0)\to H_{*}(K_1)\to\dots\to H_{*}(K_N) \to \cdots .
\] 
Such a diagram is called a \emph{persistence module}, and it decomposes into a direct sum of indecomposable diagrams, each supported on an interval.
The collection of these intervals, called the \emph{persistence barcode}, has proven to be a powerful invariant of the filtration. 

A natural question to ask is how the barcode changes when the filtration changes.
This leads to the seminal \emph{stability theorem} of Cohen-Steiner et al.\@ \cite{MR2279866}, which asserts that passing from filtrations to barcodes is a $1$-Lipschitz map. One way to approach the stability theorem is via \emph{induced matchings}, which were introduced in \cite{MR3333456}. Given a morphism of filtrations $f_{\bullet}\colon L_{\bullet}\to K_{\bullet}$, the homology functor induces a morphism $H_{*}(f_{\bullet}) \colon H_{*}(L_{\bullet}) \to H_{*}(K_{\bullet})$ of persistence modules.
From this morphism, the induced matching construction yields a partial bijection between the barcodes of $H_{*}(L_{\bullet})$ and $H_{*}(K_{\bullet})$, which can be used to bound the distance between these two barcodes from above.
The induced matching 
is defined in terms of the barcode of $\im H_{*}(f_{\bullet})$,
motivating the problem of computing this barcode.
A first algorithm for this problem has been given by Cohen-Steiner et al.\@ \cite{MR2807543} for the special case where $f_{\bullet}$ is of the form $L_{\bullet} = K_{\bullet} \cap L \hookrightarrow K_{\bullet}$.
In addition to their algorithm for image persistence, Cohen-Steiner et al.~\cite{MR2807543} present algorithms for computing barcodes of the kernel and cokernel of the morphism $H_{*}(f_{\bullet})$.
All of their algorithms rely on the standard reduction of boundary matrices. 

Cohen-Steiner et al.\@ propose applications of image persistence for recovering the persistent homology of a noisy function on a noisy domain, see also the related work by \textcite{MR2846177}.
Very recently, \textcite{reani2021cycle} proposed a method that includes the computation of induced matchings in order to pair up common topological features in different data sets, with applications to statistical bootstrapping. 
Furthermore, the computation of image barcodes is used in a distributed algorithm for persistent homology based on the Mayer--Vietoris spectral sequence by \textcite{casas2020distributing}.

Despite the usefulness of image persistence, there are a few aspects that have prevented these techniques from being widely used in applications so far.
Specifically, to the best of our knowledge, there is no publicly available implementation at this moment.
Furthermore, computation using the known algorithms is slow in comparison to modern algorithms for a single filtration.
Indeed, computing usual persistent homology for larger data sets arising in real-world applications only became feasible in recent years due to optimizations that exploit various structural properties and algebraic identities of the problem \cite{Chen.2011,MR2854319,MR4298669}.
One of our motivations is to develop a theory allowing for the adaption of these speed-ups to the computation of images and induced matchings.

One of the most important improvements for barcode computations relies on the use of cohomology based algorithms.
These were first studied by de Silva et al.\@ in \cite{MR2854319} and justified by certain duality results.
In summary, the authors provide correspondences between the barcodes for persistent homology and for \emph{persistent cohomology}, 
as well as the barcodes for \emph{persistent relative homology} 
\[
H_{*}(K,K_{\bullet})\colon H_{*}(K,K_0)\to H_{*}(K,K_1)\to\dots\to H_{*}(K,K)
\]
and similarly for \emph{persistent relative cohomology}.
The homology persistence modules simply have the same barcode as their cohomology counterparts \cite[Proposition 2.3]{MR2854319}. For the absolute-relative correspondence \cite[Proposition 2.4]{MR2854319}, it turns out that the bounded intervals in the barcodes of $H_{d-1}(K_{\bullet})$ and $H_{d}(K,K_{\bullet})$ 
are also exactly the same, and there is a one-to-one correspondence between intervals of the form $[a,\infty)$ in the barcode of $H_{d}(K_{\bullet})$ and intervals of the form $(-\infty,a)$ in the barcode of $H_{d}(K,K_{\bullet})$.

The original proof for the absolute-relative correspondence uses a decomposition of filtered chain complexes.
This strategy relies on a non-canonical choice, which does not extend to the functorial setting.
We thus adopt a different point of view based on the long exact sequence of a pair in homology.  
Applying this functorial construction to a filtration $K_{\bullet}$,
we obtain a long exact sequence of persistence modules
\[
\begin{tikzcd}
\cdots \arrow[r] & \Delta H_d(K) \arrow[r, "\limCounit_d"] & H_d(K,K_{\bullet}) \arrow[r, "\partial"] & H_{d-1}(K_{\bullet}) \arrow[r,"\colimUnit_{d-1}"] & \Delta H_{d-1}(K) \arrow[r] & \cdots 
\end{tikzcd}
\]
where $\Delta$ is used to denote constant persistence modules. 
As it turns out, 
the first two of the short exact sequences 
\[
\begin{tikzcd}[row sep=0ex]
0\arrow[r]&\im\partial\arrow[r]& H_{d-1}(K_{\bullet})\arrow[r]&\im\colimUnit_{d-1}\arrow[r]& 0
\\
0\arrow[r]&\im\limCounit_{d}\arrow[r]& H_{d}(K,K_{\bullet})\arrow[r]&\im\partial\arrow[r]& 0
\\
0\arrow[r]&\im\colimUnit_{d}\arrow[r]& \Delta H_d(K)\arrow[r]&\im\limCounit_{d}\arrow[r]& 0
\end{tikzcd}
\]
split (as a special case of \cref{cor:lifespan_split}),
showing that $\im\partial$ is a summand of both $H_{d}(K,K_{\bullet})$ and $H_{d-1}(K_{\bullet})$. Its barcode consists of the bounded intervals of either persistence module.
Moreover, the third short exact sequence has a constant persistence module $\Delta H_d(K)$ at the middle, implying that the persistence modules $\im\colimUnit_{d}$ and $\im\limCounit_{d}$ determine each other.
Together this shows that 
the barcodes of $H_{*}(K_{\bullet})$ and $H_{*}(K,K_{\bullet})$ completely determine each other. 
This argument will be made more precise in \cref{subsec:absrel}.

By observing that $\Delta H_{d}(K)\cong\Delta\colim H_{d}(K)\cong\Delta\lim H_{d}(K,K_{\bullet})$ and that $\limCounit$ and $\colimUnit$ are the counit and the unit of the adjunctions $\Delta\dashv\lim$ and $\colim\dashv\Delta$, respectively, we can generalize the definitions of the morphisms $\limCounit$ and $\colimUnit$ to arbitrary persistence modules, indexed by arbitrary totally ordered sets. Taking images, kernels and cokernels of the morphisms  $\limCounit$ and $\colimUnit$ yields endofunctors on the category of persistence modules, which we call \emph{lifespan functors}.

This general definition of lifespan functors also works in the category of \emph{matching diagrams}, since these diagrams admit limits and colimits (\cref{prop:mch_lim_colim}). The category of matching diagrams is equivalent to the category of barcodes \cite{Bauer2020Persistence}. 
The effect of the lifespan functors on persistence modules and matching diagrams is best described in terms of barcodes, as for the above example $\im\partial \cong \ker\colimUnit_{d-1} \cong \coker\limCounit_{d}$, whose barcode corresponds to the bounded intervals; see also \cref{fig:lifespan} for an illustration and \cref{cor:change_in_barcode} for a precise statement.

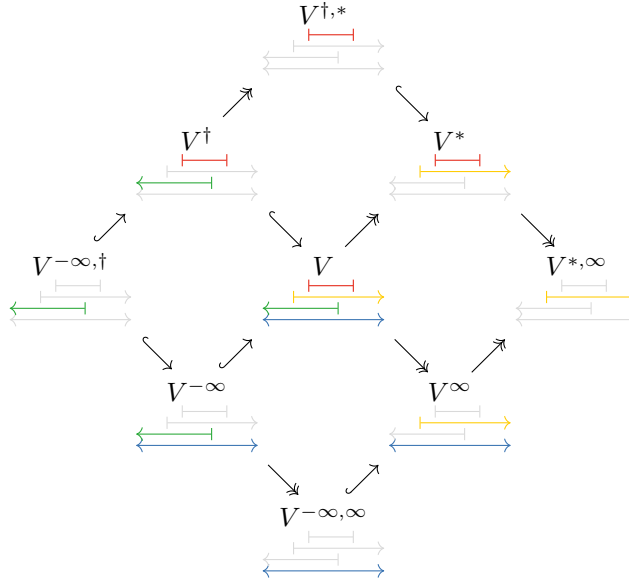
\begin{figure}[h!]
\small
\centering
\begin{tikzcd}[row sep={11ex,between origins},column sep={11ex,between origins}]
&&
\underset{\begin{tikzpicture}[xscale=0.2,yscale=0.15,anchor=north]
\draw[mygray,<->] (0,0) -- (8,0);
\draw[mygray,<-|] (0,1) -- (5,1);
\draw[mygray,|->] (2,2) -- (8,2);
\draw[myred,|-|] (3,3) -- (6,3);
\end{tikzpicture}}
{\finitePart{V}}\arrow[dr,hook]&&\\
&
\underset{\centering\begin{tikzpicture}[xscale=0.2,yscale=0.15,anchor=north]
\draw[mygray,<->] (0,0) -- (8,0);
\draw[mygreen,<-|] (0,1) -- (5,1);
\draw[mygray,|->] (2,2) -- (8,2);
\draw[myred,|-|] (3,3) -- (6,3);
\end{tikzpicture}}
{\deathPart{V}} \arrow[ur,two heads]\arrow[rd, hook] & &
\underset{\centering\begin{tikzpicture}[xscale=0.2,yscale=0.15,anchor=north]
\draw[mygray,<->] (0,0) -- (8,0);
\draw[mygray,<-|] (0,1) -- (5,1);
\draw[myyellow,|->] (2,2) -- (8,2);
\draw[myred,|-|] (3,3) -- (6,3);
\end{tikzpicture}}
{\birthPart{V}}
\arrow[dr,two heads] &\\
\underset{\centering\begin{tikzpicture}[xscale=0.2,yscale=0.15,anchor=north]
\draw[mygray,<->] (0,0) -- (8,0);
\draw[mygreen,<-|] (0,1) -- (5,1);
\draw[mygray,|->] (2,2) -- (8,2);
\draw[mygray,|-|] (3,3) -- (6,3);
\end{tikzpicture}}
{\ancientDeathPart{V}}
\arrow[ur,hook]\arrow[dr,hook]&& 
\underset{\centering\begin{tikzpicture}[xscale=0.2,yscale=0.15,anchor=north]
\draw[myblue,<->] (0,0) -- (8,0);
\draw[mygreen,<-|] (0,1) -- (5,1);
\draw[myyellow,|->] (2,2) -- (8,2);
\draw[myred,|-|] (3,3) -- (6,3);
\end{tikzpicture}}
{V}
\arrow[ru, two heads] \arrow[rd, two heads] &&
\underset{\centering\begin{tikzpicture}[xscale=0.2,yscale=0.15,anchor=north]
\draw[mygray,<->] (0,0) -- (8,0);
\draw[mygray,<-|] (0,1) -- (5,1);
\draw[myyellow,|->] (2,2) -- (8,2);
\draw[mygray,|-|] (3,3) -- (6,3);
\end{tikzpicture}}
{\immortalBirthPart{V}}
\\
&
\underset{\centering\begin{tikzpicture}[xscale=0.2,yscale=0.15,anchor=north]
\draw[myblue,<->] (0,0) -- (8,0);
\draw[mygreen,<-|] (0,1) -- (5,1);
\draw[mygray,|->] (2,2) -- (8,2);
\draw[mygray,|-|] (3,3) -- (6,3);
\end{tikzpicture}}
{\ancientPart{V}}
\arrow[dr,two heads] \arrow[ru, hook] & &
\underset{\centering\begin{tikzpicture}[xscale=0.2,yscale=0.15,anchor=north]
\draw[myblue,<->] (0,0) -- (8,0);
\draw[mygray,<-|] (0,1) -- (5,1);
\draw[myyellow,|->] (2,2) -- (8,2);
\draw[mygray,|-|] (3,3) -- (6,3);
\end{tikzpicture}}
{\immortalPart{V}}\arrow[ur,two heads]&\\
&&
\underset{\centering\begin{tikzpicture}[xscale=0.2,yscale=0.15,anchor=north]
\draw[myblue,<->] (0,0) -- (8,0);
\draw[mygray,<-|] (0,1) -- (5,1);
\draw[mygray,|->] (2,2) -- (8,2);
\draw[mygray,|-|] (3,3) -- (6,3);
\end{tikzpicture}}
{\constantPart{V}}
\arrow[ur,hook]&&
\end{tikzcd}
\caption{Lifespan functors applied to 
 an $\mathbb R$-indexed persistence module $V$
with a visualization of their effect on the barcode, according to \cref{cor:change_in_barcode}.}
\label{fig:lifespan}
\end{figure}

As an application of lifespan functors, we give a simple characterization of the projective and injective objects both in the category of barcodes (or matching diagrams) as well as the category of \emph{pointwise finite-dimensional} persistence modules (\cref{prop:inj_proj_barc,thm:proj_inj}).
In both cases, the projective objects are those with vanishing \emph{mortal part} $\deathPart{(-)}$, while the injective objects are those with vanishing \emph{nascent part} $\birthPart{(-)}$.

The lifespan functors also allow us to succinctly express the above absolute-relative correspondence in terms of certain natural isomorphisms and correspondences (\cref{prop:abs_rel_single_filtration}):
 \begin{align*}
 \deathPart{H_{d-1}(K_\bullet)} &\cong \im\partial \cong \birthPart{H_d(K,K_\bullet)} ,
 &
 \immortalPart{H_{d}(K_\bullet)} \cong \im\colimUnit_{d}  &\leftrightarrow \im\limCounit_{d} \cong \ancientPart{H_{d}(K,K_\bullet)}.
\end{align*}
Here, naturality is inherited from the construction of the long exact homology sequence.
In particular,
for the morphism of filtrations $f_{\bullet}\colon L_{\bullet}\to K_{\bullet}$ induced by a map $f\colon L\to K$,
we get an isomorphism
$\deathPart{H_{d-1}(f_{\bullet})} \cong \birthPart{H_d(f, f_{\bullet})}$.
We also get a morphism 
\[
\begin{tikzcd}
\cdots \arrow[r] & \Delta H_d(K) \arrow[r, "\limCounit_d"] & H_d(K,K_{\bullet}) \arrow[r, "\partial"] & H_{d-1}(K_{\bullet}) \arrow[r,"\colimUnit_{d-1}"] & \Delta H_{d-1}(K) \arrow[r] & \cdots \\
\cdots \arrow[r] & \Delta H_d(L)\arrow[u] \arrow[r, "\limCounit_d"] & H_d(L,L_{\bullet}) \arrow[u]\arrow[r, "\partial"] & H_{d-1}(L_{\bullet}) \arrow[u]\arrow[r,"\colimUnit_{d-1}"] & \Delta H_{d-1}(L)\arrow[u] \arrow[r] & \cdots 
\end{tikzcd}
\]
of long exact sequences. 
Note, however, that the induced sequences of kernels, images, and cokernels are no longer exact in general, so the rest of the proof of the absolute-relative-correspondence for a single filtration does not carry over completely to this setting.
In order to still obtain a useful absolute-relative correspondences for $H_*(f)$,
we develop conditions for when the lifespan functors commute with taking images, kernels, and cokernels of morphisms (\cref{thm:image_parts}), so that, for example, we get
\[
\deathPart{(\im H_{d-1}(f_{\bullet}))} \cong \im \big(\deathPart{H_{d-1}(f_{\bullet})}\big)
\cong \im \big(\birthPart{H_d(f, f_{\bullet})}\big) \cong \birthPart{(\im H_d(f, f_{\bullet}))},
\]
meaning that, as in the single filtration case, the bounded intervals in the barcodes of the images of the absolute and relative morphism agree. Furthermore, we will also state a functorial version of the correspondence between persistent homology and cohomology in terms of vector space duality (\cref{prop:hom_cohom}) and analyze how the lifespan functor behave with respect to dualization (\cref{subsec:lifespan_dual}). 

As explained in \cite{MR4298669}, computing cohomology instead of homology is particularly relevant in conjunction with the \emph{clearing} optimization, introduced %
by Chen and Kerber in \cite{Chen.2011},
and also used implicitly in the cohomology algorithm by de Silva et al.\@ \cite{MR2854319}.
There is also an adaptation of this optimization for the computation of barcodes of images, which we aim to formalize in future work. We already provide an implementation of the resulting method for computing barcodes of images of maps between the persistent homologies of Vietoris-Rips filtrations \cite{ripser_image}.

\subsubsection*{Outline}
We start off by reviewing some preliminaries on persistence including the theory of matching diagrams in \cref{sec:prelim}. In \cref{sec:lifespan}, we present the lifespan functors and some of their relevant properties in a general setting. The lifespan functors are then specialized to persistence modules in \cref{sec:lifespan_persistence}. As an application, we then use the lifespan functors in \cref{subsec:proj_inj} to classify injective and projective objects in the categories of barcodes and p.f.d.\@ persistence modules. We finish by proving functorial dualities in persistent homology involving our lifespan functors in \cref{sec:dualities}.

\subsubsection*{Notation and Conventions}
Throughout the paper, we fix a totally ordered set $(T,\leq)$ and write $\mathbf{T}$ for the corrseponding category. Note that $\mathbf{T^{\mathrm{op}}}$ is then the category corresponding to $(T,\geq)$. We also fix a field $\mathbb{F}$ and write $\mathbf{Vec}$ for the category of vector spaces over $\mathbb{F}$. The full subcategory of finite dimensional vector spaces is denoted by $\mathbf{vec}$. We write $\mathbf{Top}$ for the category of topological spaces. If $\mathbf{A}$ is an abelian category, we write $\mathbf{Ch(A)}$ for the category of chain complexes in $\mathbf{A}$. If $\mathbf{C}$ and $\mathbf{D}$ are categories, we write $\mathbf{C^D}$ for the category of functors $\mathbf{C}\to\mathbf{D}$. 

\section{Preliminaries}\label{sec:prelim}
We recall the basic definitions of absolute and relative persistent (co)homology in \cref{subsec:pers_hom}. They are examples of persistence modules, which we go over together with matching diagrams in \cref{subsec:pers_mod}. Subsequently, we present a formal framework for barcodes in \cref{subsec:barcodes}. The equivalence between barcodes and matching diagrams is recalled in \cref{subsec:equivalence}. We then collect some basic categorical properties of barcodes and matching diagrams in \cref{subsec:mchdgm}. Finally, we recall some facts about dualization of persistence modules in \cref{subsec:dualization}.

\subsection{Persistent Homology}\label{subsec:pers_hom}
\begin{defi}\label{defi:persistence_module}
The category of \emph{persistence modules indexed by $\mathbf{T}$} is defined as the functor category $\mathbf{Vec^T}$. The category of \emph{pointwise finite dimensional (p.f.d.\@) persistence modules indexed by $\mathbf{T}$} is defined as $\mathbf{vec^T}$.
\end{defi}

The full subcategory $\mathbf{vec}$ is closed under taking kernels, cokernels and finite direct sums in the abelian category $\mathbf{Vec}$, so it is also abelian. 
Moreover, since $\mathbf{T}$ is a small category, the functor categories $\mathbf{vec^T}$ and $\mathbf{Vec^T}$ are again abelian, with kernels, cokernels, direct sums, etc.\  given pointwise. 

The most commonly studied example of a persistence module is the persistent homology of a diagram of spaces. To define it, we start by observing that there is a purely formal identification of the categories $\mathbf{Ch(Vec^T)}$ and $\mathbf{Ch(Vec)^T}$. Thus, if we have a chain complex of persistence modules, we can interpret it as a diagram of chain complexes, and vice versa. Objects in these identified categories will be called \emph{persistent chain complexes indexed by $\mathbf{T}$}. %

Let $C_*\colon\mathbf{Top}\to\mathbf{Ch(Vec)}$ be the functor assigning to a topological space its singular chain complex with coefficients in $\mathbb{F}$. Similarly, let $C^* 
\colon\mathbf{Top}\to\mathbf{Ch(Vec)}^{\mathrm{op}}$ denote the singular cochain complex functor with coefficients in $\mathbb{F}$.

If $X\colon\mathbf{T}\to\mathbf{Top}$ is a filtration, or any $\mathbf{T}$-indexed diagram of topological spaces, composing with $C_*$ and $C^{*}$ yields a persistent chain complex $C_*(X)$ indexed by $\mathbf{T}$ and a persistent cochain complex $C^{*}(X)$ indexed by $\mathbf{T^{\mathrm{op}}}$. 
Moreover, recall that every diagram $X$ in $\mathbf{Top}$ has a colimit, and
the natural map $X\to\Delta\colim X$ induces morphisms $C_{*}(X)\to C_{*}(\Delta\colim X)$ and $C^*(\Delta\colim X)\to C^*(X)$. 
We write
\begin{align*}
C_{*}(\colim X,X)&=\coker(C_{*}(X)\to C_{*}(\Delta\colim X)),\\
C^{*}(\colim X,X)&=\ker(C^*(\Delta\colim X)\to C^*(X)).
\end{align*}
We then define the \emph{$d$-th persistent homology} of $X$ as $H_{d}(X)=H_{d}(C_{*}(X))$ and the \emph{$d$-th persistent homology} of $X$ as $H_{d}(X)=H_{d}(C_{*}(X))$, and similarly for cohomology.

Note that the relative versions are intrinsic to the diagram $X$, and that persistent homology is indexed by $\mathbf{T}$, while persistent cohomology is indexed by $\mathbf{T^{\mathrm{op}}}$.

\subsection{Structure of Persistence Modules}\label{subsec:pers_mod}

A natural way to construct persistence modules is to specify a basis for each index $t \in T$ and then to specify the linear maps by matching the basis elements of different indices in a compatible way.
This can be formalized as a functor from a category of sets and matchings to the category of vector spaces. 
In fact, this construction already generates all possible p.f.d.\@ persistence modules up to isomorphism, providing a structure theorem for p.f.d.\@ persistence modules.
We now introduce the requisite definitions and recall the fundamental results from the literature.

\begin{defi}\label{defi:matchings}
If $A$ and $B$ are sets, a subset $\sigma\subseteq A\times B$ is called a \emph{matching} if for each $a\in A$ there is at most one $b\in B$ with $(a,b)\in\sigma$ and for each $b\in B$ there is at most one $a\in A$ with $(a,b)\in\sigma$. If $\tau\subseteq B\times C$ is another matching, we define the composition $\tau\circ\sigma\subseteq A\times C$ as 
\[
\tau\circ\sigma=\{(a,c)\mid\text{ there exists }b\in B\text{ with }(a,b)\in\sigma\text{ and }(b,c)\in\tau\}.
\]
The resulting category, with sets as objects, matchings as morphisms, and the above composition, will be denoted by $\mathbf{Mch}$. If $\sigma\subseteq A\times B$ is a matching, we define its \emph{opposite} matching
\[
\sigma'=\{(b,a)\mid (a,b)\in\sigma\}\subseteq B\times A.
\] 
This construction makes the category $\mathbf{Mch}$ self-dual, i.e., it yields an isomorphism between $\mathbf{Mch}$ and its opposite category.
We define the category of \emph{matching diagrams indexed by $\mathbf{T}$} as the functor category $\mathbf{Mch^T}$. 
\end{defi}

We can now functorially assign a persistence module to a matching diagram.

\begin{defi}\label{defi:matching_module}
We define the functor $\mathcal{F}\colon\mathbf{Mch}\to\mathbf{Vec}$ by sending a set $A$ to the free vector space generated by $A$ and sending a matching $\sigma\subseteq A\times B$ to the linear extension of the map 
\[
a\mapsto
\begin{cases}
    b & \text{if }(a,b)\in\sigma\\
    0              & \text{otherwise.}
\end{cases}
\]
We define the \emph{matching module} functor $\mathcal{F}\colon\mathbf{Mch}^\mathbf{T}\to\mathbf{Vec^T}$ by applying $\mathcal{F}$ pointwise, i.e., $\mathcal{F}(D)=\mathcal{F}\circ D$.
\end{defi}

In certain cases we can also go from persistence modules to matching diagrams.

\begin{defi}\label{defi:interval_decomp}
We say that a persistence module $M$ is \emph{interval-decomposable} if there exists a matching diagram $D$ with $\mathcal{F}(D)\cong M$. A choice of isomorphism $\mathcal{F}(D) \cong M$ is called an \emph{interval decomposition}. %
\end{defi}

Interval decompositions are typically described in terms of barcodes, which are collections of intervals in $\mathbf{T}$.
We will formally introduce barcodes in the next section, and develop their equivalence to matching diagrams $\mathbf{Mch}^\mathbf{T}$ in a categorical sense.

The most commonly used existence result
asserts that every p.f.d.\ persistence module admits an interval decomposition \cite[Theorem~1.1]{MR3323327}; see also \cite[Theorem~1.2]{MR4143378}.
While there may be many different interval decompositions for a single persistence module, by a version of the Krull--Schmidt--Azumaya Theorem the structure of the decomposition is still unique, which can be conveniently phrased in the language of matching diagrams as follows.

\begin{thm}[\cite{MR0265428}, Theorem~2.7; see also \cite{MR0265428}, Section 4.8]\label{thm:krull_schmidt}
The functor $\mathcal{F} \colon \mathbf{Mch}^\mathbf{T} \to \mathbf{Vect}^\mathbf{T}$ reflects the property of being isomorphic:
If $D$ and $D'$ are matching diagrams with $\mathcal{F}(D)\cong\mathcal{F}(D')$, then already $D\cong D'$.

\end{thm}

\subsection{Barcodes}\label{subsec:barcodes}

An equivalent description for a matching diagram can be given in terms of a collection of intervals, called barcode.
The intervals encode the index range of matched elements in the matching diagram.
A barcode should be thought of as a multiset of intervals, that is, the same interval may appear multiple times. 

\begin{defi}\label{defi:barcodes}
We denote the set of all intervals in $T$ as $\mathfrak{I}(T)$, or simply as 
$\mathfrak{I}$ when the index set is clear from the context.
If $A$ is an arbitrary set, we call any subset $B\subseteq\mathfrak{I}\times A$ a \emph{barcode in $T$}. 
\end{defi}
The purpose of the set $A$ in this definition is to distinguish multiple instances of the same interval, as in the standard construction of a disjoint union.
If clear from the context, we sometimes suppress this index from the notation.
If $B\subseteq\mathfrak{I}\times A$ is a barcode and $I$ an interval in~$T$, the cardinality of the set 
$\{a\in A\mid (I,a)\in B\}$
measures how many copies of $I$ are in $B$. 

Barcodes form a category, which is equivalent to the category of matching diagrams \cite{Bauer2020Persistence}. 
We introduce some terminology used to give an explicit description of the morphisms in the category of barcodes.

\begin{defi}\label{defi:interval_language}
Let $I$ and $J$ be intervals in $T$. We say that \emph{$I$ bounds $J$ above} if for all $s\in J$ there exists $t\in I$ such that $s\leq t$. We say that \emph{$I$ bounds $J$ below} if for all $u\in J$ there exists $t\in I$ such that $t\leq u$. %
We say that \emph{$I$ overlaps $J$ above}, or that \emph{$J$ overlaps $I$ below}, if their intersection is non-empty, $I$ bounds $J$ above, and $J$ bounds $I$ below.
\end{defi}

\begin{defi}\label{defi:barcodes_category}
For barcodes $B$ and $B'$, we call a matching $\sigma\subseteq B\times B'$ an \emph{overlap matching} if for each $((I,a),(I',a'))\in\sigma$ the interval $I$ overlaps the interval $I'$ above. If $\sigma\subseteq B\times B'$ and $\tau\subseteq B'\times B''$ are overlap matchings, we define their \emph{overlap composition} as 
\[
\tau\bullet\sigma=\{((I,a),(I'',a''))\in\tau\circ\sigma\mid I\text{ overlaps }I''\text{ above}\}.
\]
The resulting category with barcodes as objects, overlap matchings as morphisms and overlap composition will be denoted by $\mathbf{Barc(T)}$.
\end{defi}

Note that two barcodes $B\subseteq\mathfrak{I}\times A$ and $B'\subseteq\mathfrak{I}\times A'$ are isomorphic if and only if there is a bijection $f\colon B\to B'$ such that for all $(I,a)\in B$ there is $a'$ in $A'$ with $f(I,a)=(I,a')$. In other words, $B$ and $B'$ are isomorphic if and only if the sets $\{a\in A\mid (I,a)\in B\}$ and $\{a'\in A'\mid (I,a')\in B'\}$ have the same cardinality for every interval $I\in\mathfrak{I}$.

\subsection{Equivalence of Barcodes and Matching Diagrams}\label{subsec:equivalence}
As we have mentioned before, the two categories $\mathbf{Mch^T}$ and $\mathbf{Barc(T)}$ are equivalent. We will now review the construction of an explicit equivalence following \cite{Bauer2020Persistence}.

\begin{defi}\label{defi:component_set}
Let $D$ be a matching diagram. We define its \emph{components} as the set of equivalence classes
\[
\mathcal{C}(D)={\textstyle \left(\bigcup\limits_{t\in T}\{t\}\times D_{t}\right)} / {\textstyle \sim},
\]
where the equivalence relation $\sim$ is defined as follows: For $d\in D_{t}$ and $d'\in D_{u}$ we set $(t,d)\sim (u,d')$ if and only if $(d,d')\in D_{t,u}$ or $(d',d)\in D_{u,t}$. 
Note that each component $Q\in\mathcal{C}(D)$ can also be regarded as a matching diagram such that $Q_{t}\subseteq D_{t}$ has at most one element for each $t \in T$.
For a component $Q\in\mathcal{C}(D)$, we define its \emph{support} as the range of indices in $T$ spanned by the component,
\[
\supp(Q)=\{t\in T\mid (t,d)\in Q \text{ for some } d \in D_t \}.
\]

\end{defi}

We use these notions to construct a functor from matching diagrams to barcodes.

\begin{defi}\label{defi:B}
We define a functor $\mathcal{B}\colon\mathbf{Mch^T}\to\mathbf{Barc(T)}$ by setting 
\[
\mathcal{B}(D)=\{(I,Q)\in \mathfrak{I}\times\mathcal{C}(D)\mid I=\supp(Q)\}
\]
for any matching diagram $D$ and
\[
\mathcal{B}(\psi)=\{((I,Q),(I',R))\in\mathcal{B}(D)\times\mathcal{B}(E)\mid
Q_{t}\times R_{t} \subseteq \psi_{t} \text{ for all }t\in I \cap I'
\}
\]
for any morphism of matching diagrams $\psi\colon D\to E$.
\end{defi}

As shown in \cite{Bauer2020Persistence}, the support of a component is indeed an interval, a morphism of matching diagrams is mapped to an overlap matching by the above construction, and we indeed get a functor.
Conversely,
we can also pass from barcodes to matching diagrams.

\begin{defi}\label{defi:D}
We define a functor $\mathcal{D}\colon\mathbf{Barc(T)}\to\mathbf{Mch^T}$ by setting $\mathcal{D}(B)$ for any barcode $B$ to be the matching diagram given by 
\begin{align*}
\mathcal{D}(B)_{t}   & = \{(I,a)\in B\mid t\in I\} , \\
\mathcal{D}(B)_{t,u} & = \{((I,a),(I',a'))\in \mathcal{D}(B)_{t}\times\mathcal{D}(B)_{u}\mid (I,a)=(I',a')\}.
\end{align*}
For an overlap matching $\sigma$, we let $\mathcal{D}(\sigma)$ be the morphism of matching diagrams with
\[
\mathcal{D}(\sigma)_{t}=\{((I,a),(I',a'))\in\sigma\mid t\in I\cap I'\}.
\]
\end{defi}
Again, we refer to \cite{Bauer2020Persistence} for the fact that $\mathcal{D}$ is a well-defined functor.

\begin{thm}[\cite{Bauer2020Persistence}]\label{thm:equivalence}
The functors $\mathcal{B}\colon\mathbf{Mch^T}\to\mathbf{Barc(T)}$ and $\mathcal{D}\colon\mathbf{Barc(T)}\to\mathbf{Mch^T}$ defined above are quasi-inverses. In particular, the categories $\mathbf{Mch^T}$ and $\mathbf{Barc(T)}$ are equivalent.
\end{thm}
Note that in \cite{Bauer2020Persistence}, the equivalences were denoted by $E$ and $F$. 
Using this equivalence, we can give an explicit description of the composite functor $\mathcal{F}\circ\mathcal{D} \colon \mathbf{Barc(T)}\to\mathbf{Vec^T}$ constructing a persistence module with a given barcode as follows.

\begin{defi}\label{defi:barcode_module}
Let $I\subseteq T$ be an interval. The \emph{interval module} $C(I)$ is
the persistence module obtained from the barcode consisting of a single instance of $I$:
 \[
C(I)_t=
\begin{cases}
    \mathbb{F} & \text{if } t\in I,\\
    0              & \text{otherwise,}
\end{cases}
\qquad
\text{with structure maps}
\qquad
C(I)_{t,u}=
\begin{cases}
    \id_{\mathbb{F}} & \text{if } t,u\in I,\\
    0              & \text{otherwise.}
\end{cases}
\]
If $I$ and $J$ are intervals such that $I$ overlaps $J$ above, there exists a canonical morphism $\varphi(I,J)\colon C(I)\to C(J)$ defined by 
\[
\varphi(I,J)_t=
\begin{cases}
    \id_{\mathbb{F}} & \text{if } t\in I\cap J,\\
    0              & \text{otherwise.}
\end{cases}
\] 
We define the \emph{barcode module functor} $\mathcal{M}\colon\mathbf{Barc(T)}\to\mathbf{Vec^T}$ 
by sending a barcode $B$ to the direct sum of interval modules
\(
\bigoplus_{(I,a)\in B}C(I)
\)
and sending an overlap matching $\sigma\subseteq B\times B'$ to the direct sum of the morphisms $\varphi(I,I') \colon C(I) \to C(I')$ for all pairs $((I,a),(I',a'))\in\sigma$.
If a persistence module $M$ satisfies $\mathcal{M}(B)\cong M$, we say that $B$ is a \emph{barcode of $M$}.
\end{defi}

The following proposition is straightforward to verify from the definitions.

\begin{prop}\label{prop:passing_to_modules}
There are natural isomorphisms $\mathcal{F}\cong\mathcal{M}\circ\mathcal{B}$ and $\mathcal{M}\cong\mathcal{F}\circ\mathcal{D}$.
\end{prop}

\subsection{Categorical Properties of Matching Diagrams}\label{subsec:mchdgm}

One can use large parts of the theory of homological algebra in the categories $\mathbf{Mch}$ and $\mathbf{Mch^T}$ since they have the following property.

\begin{defi}\label{defi:puppe_exact}
A category is called \emph{Puppe-exact} or \emph{p-exact} if
	it has a zero object,
	it has all kernels and cokernels,
	every mono is a kernel and every epi is a cokernel, and
	every morphism has an epi-mono-factorization.
\end{defi}
Put informally, a Puppe-exact category is an abelian category that need not have (co)products. Recall that in any category with kernels and cokernels, monos have vanishing kernels and epis have vanishing cokernels. While the converse is not true in general, it is true in p-exact categories.

\begin{lem}[{\cite[Korollar 2.4.4]{MR0269713}}]
A morphism in a p-exact category is mono if and only if its kernel vanishes and it is epi if and only if its cokernel vanishes.
\end{lem}

We will use this lemma throughout without explicit reference. In particular, we will use it for barcodes and matching diagrams, which form p-exact categories.

\begin{prop}[{\cite[Section~1.6.4]{MR2977522}}]\label{thm:mch_puppe_exact}
$\mathbf{Mch}$ is Puppe-exact. For a matching $\sigma\subseteq A\times B$ we have
\begin{align*}
\ker\sigma&=\{a\in A\mid (a,b)\notin\sigma\text{ for all }b\in B\},\\
\im\sigma&=\{b\in B\mid (a,b)\in\sigma\text{ for some }a\in A\},\\
\coker\sigma&=\{b\in B\mid (a,b)\notin\sigma\text{ for all }a\in A\},\\
\coim\sigma&=\{a\in A\mid (a,b)\in\sigma\text{ for some }b\in B\}
.
\end{align*}
$\mathbf{Mch^T}$ is also Puppe-exact, with kernels, cokernels etc.\ given pointwise.
\end{prop}

Importantly, these constructions are compatible with the passage to vector spaces and persistence modules, as expressed in the following statement.

\begin{prop}\label{prop:F_creates_exactness}
The functor $\mathcal{F}$ preserves and reflects exactness, i.e., a sequence of matchings 
$V \to V' \to V''$
is exact if and only if the corresponding sequence of vector spaces 
$\mathcal{F}(V) \to \mathcal{F}(V') \to \mathcal{F}(V'')$
is exact. The same holds for $\mathcal{F}$ as a functor $\mathbf{Mch^T} \to \mathbf{Vec^T}$.
\end{prop}

Using the equivalence between $\mathbf{Mch^T}$ and $\mathbf{Barc(T)}$, we can translate the constructions in \cref{thm:mch_puppe_exact} to describe the kernels, cokernels, and images of overlap matchings explicitly as barcodes.

\begin{defi}\label{defi:ker_coker_intervals}
For an overlap matching $\sigma\subseteq B\times B'$ and $(I,a)\in B$, $(I',a')\in B'$, we set
\begin{align*}
\ker(\sigma,(I,a))&=
\begin{cases}
    (I\setminus I',a) & \text{if }((I,a),(I',a'))\in\sigma , \\
    (I,a)              & \text{otherwise} ;
\end{cases}
\\
\coker(\sigma,(I',a'))&=
\begin{cases}
    (I'\setminus I,a') & \text{if }((I,a),(I',a'))\in\sigma , \\
    (I',a')              & \text{otherwise} .
\end{cases}
\end{align*}
\end{defi} 

\begin{prop}[\cite{Bauer2020Persistence}]\label{prop:ker_overlap_matching}
Let $B\subseteq\mathfrak{I}\times A$ and $B'\subseteq\mathfrak{I}\times A'$ be barcodes. Any overlap matching $\sigma\subseteq B\times B'$ has a kernel, coimage, image and cokernel in $\mathbf{Barc(T)}$, with
\begin{align*}
\ker\sigma&=\{(J,a)\in\mathfrak{I}\times A\mid J=\ker(\sigma,(I,a))\text{ for }(I,a)\in B\}\\
\coim\sigma&=\{(J,a)\in\mathfrak{I}\times A\mid J=I\cap I'\text{ for }((I,a),(I',a'))\in\sigma\}\\
\im\sigma&=\{(J,a')\in\mathfrak{I}\times A'\mid J=I\cap I'\text{ for }((I,a),(I',a'))\in\sigma\}\\
\coker\sigma&=\{(J,a')\in\mathfrak{I}\times A'\mid J=\coker(\sigma,(I',a'))\text{ for }(I',a')\in B'\}.
\end{align*}
\end{prop}

Using the p-exact structure on barcodes, we will later consider exact sequences of barcodes and translate them to exact sequences of persistence modules. We will also further study the categorical structure on barcodes later on and classify injective and projective objects. In both of these settings, the following characterization of split mono and epi overlap matchings will be important.

\begin{lem}\label{prop:split_overlap}
Let $\sigma$ be an overlap matching and assume that $\sigma$ is mono or epi. Then $\sigma$ is split if and only if $((I,a),(I',a'))\in\sigma$ implies $I=I'$.
\end{lem}
\begin{proof}
If $((I,a),(I',a'))\in\sigma$ implies $I=I'$ for some overlap matching $\sigma$, then its opposite matching $\sigma'$ is again an overlap matching. If $\sigma$ is epi, this yields a right inverse and if $\sigma$ is mono, this yields a left inverse.

If on the other hand $\sigma$ is split mono or split epi, there needs to be an overlap matching $\tau$ with $((I',a'),(I,a))\in\tau$ whenever $((I,a),(I',a'))\in\sigma$. Since both $\sigma$ and $\tau$ are overlap matchings, $I$ and $I'$ overlap each other above, so we have $I=I'$ whenever $((I,a),(I',a'))\in\sigma$.
\end{proof}

\subsection{Dualization for Persistence Modules}\label{subsec:dualization}
We have seen persistent homology and cohomology as examples of persistence modules indexed by either $(T,\leq)$ or $(T,\geq)$. In general, the following yields a way of translating between the two.

\begin{defi}
We define the contravariant \emph{dualization} functor $\dual{(-)}\colon\mathbf{Vec^{T}}\to\mathbf{Vec^{T^{\mathrm{op}}}}$ by applying vector space dualization pointwise, i.e., for a $\mathbf{T}$-indexed persistence module $M$, its \emph{dual} $\dual{M}$ is the $\mathbf{T^{\mathrm{op}}}$-indexed persistence module given by $\dual{M}_{t}=\Hom(M_{t},\mathbb{F})$ for all $t\in T$.
\end{defi}

Note that a subset $I\subseteq T$ is an interval with respect to $\leq$ if and only if it is an interval with respect to $\geq$. This yields an obvious contravariant isomorphism between $\mathbf{Barc(T)}$ and $\mathbf{Barc(T^{\mathrm{op}})}$ which maps each barcode to itself. Thus, we can compare barcodes of persistence modules indexed by $(T,\leq)$ with barcodes of persistence modules indexed by $(T,\geq)$. As such, we have the following well-known fact.

\begin{lem}\label{lem:dual_barcode}
Let $M$ be a p.f.d. persistence module. Then $B$ is a barcode for $M$ if and only if it is a barcode for $\dual{M}$.
\end{lem}

Recall that $\mathbb{F}$ is injective as a module over itself, which means that the contravariant functor $\Hom(-,\mathbb{F})\colon\mathbf{Vec}\to\mathbf{Vec}$ is exact. As pointwise application of an exact functor yields an exact functor of diagram categories, we get the following.

\begin{lem}
\label{lem:dual_maps}
The dualization functor $\dual{(-)}$ is exact.
In particular, a morphism $\varphi\colon M\to N$ of persistence modules yields isomorphisms
\begin{align*}
\dual{(\ker\varphi)}&\cong\coker\dual{\varphi} , &
\dual{(\im\varphi)}&\cong\im\dual{\varphi} , &
\dual{(\coker\varphi)}&\cong\ker\dual{\varphi} .
\end{align*}
\end{lem}

\section{Lifespan Functors}\label{sec:lifespan}
In \cref{subsec:lifespan}, we will construct what we call the lifespan functors based on (co)units of the adjunction of the diagonal functor with (co)limits. We then establish conditions for when the lifespan functors commute with the image, kernel, and cokernel functors in \cref{subsec:lifespan_images}.

\subsection{Defining Lifespan Functors}\label{subsec:lifespan}
Let $\mathbf{A}$ be any category with $\mathbf{T}$-shaped limits and colimits, so that we get functors $\lim\colon\mathbf{A^T}\to\mathbf{A}$ and $\colim\colon\mathbf{A^T}\to\mathbf{A}$. As for any functor category, we also have a diagonal functor $\Delta\colon\mathbf{A}\to\mathbf{A^T}$, mapping each object to the corresponding constant diagram. Of course, this setting includes the case where $\mathbf{A}=\mathbf{Vec}$. 

For each object $V$ in $\mathbf{A^T}$, the canonical maps $V_t \to \colim V$ for $t\in T$ form a natural transformation $\colimUnit_V\colon V \to \Delta\colim V.$
Recall that $\colim$ is left adjoint to the diagonal functor $\Delta$, and the morphism $\colimUnit_V$ is the component at $V$ for the unit $\colimUnit \colon \id_{\mathbf{A^T}} \to \Delta \circ \colim$ of the adjunction $\colim\dashv\Delta$. 
Similarly, the canonical maps $\lim V \to V_t$ give a natural transformation $\limCounit_V\colon \Delta\lim V \to V,$ which is the counit $\limCounit \colon \Delta \circ \lim \to \id_{\mathbf{A^T}}$ of the adjunction $\Delta\dashv\lim$.
We thus get the diagram
\[
\begin{tikzcd}%
\Delta\lim V \ar[r,"\limCounit_V"] & V \ar[r,"\colimUnit_V"] & \Delta\colim V .
\end{tikzcd} 
\] 
From now on, we assume that $\mathbf{A}$ is Puppe-exact, so that we can form kernels, cokernels, and images.

\begin{defi}\label{defi:lifespan_1}
We define the following functors $\mathbf{A^T}\to\mathbf{A^T}$.
\begin{enumerate}
	\item The \emph{mortal part} functor is defined as $\deathPart{(-)}=\ker\colimUnit_{(-)}$.  
	\item The \emph{immortal part} functor is defined as $\immortalPart{(-)}=\im\colimUnit_{(-)}$.
	\item The \emph{nascent part} functor is defined as $\birthPart{(-)}=\coker\limCounit_{(-)}$.  
	\item The \emph{ancient part} functor is defined as $\ancientPart{(-)}=\im\limCounit_{(-)}$.
\end{enumerate}
\end{defi}
\goodbreak

By definition, for each object $V$ in $\mathbf{A^T}$ we get 
a natural diagram 
\[
\begin{tikzcd}[row sep={5ex,between origins},column sep={5ex,between origins}]
\deathPart{V}%
 \arrow[rd, hook] & & \birthPart{V}%
 \\
& V \arrow[ru, two heads] \arrow[rd, two heads] & \\
\ancientPart{V}%
\arrow[ru, hook] & & \immortalPart{V} %
\end{tikzcd} 
\] 
with diagonal short exact sequences. 
We also get composite natural transformations 
\[
\deathPart{(-)}\to\id_{\mathbf{A^T}}\to\birthPart{(-)}
\qquad
\text{and}
\qquad
\ancientPart{(-)}\to\id_{\mathbf{A^T}}\to\immortalPart{(-)}
\]
and can again form kernels, cokernels and images to get new functors.

\begin{defi}\label{defi:lifespan_2}
We define the following functors $\mathbf{A^T}\to\mathbf{A^T}$.
\begin{enumerate}
	\item The \emph{finite part} functor is defined as $\finitePart{(-)}=\im(\deathPart{(-)}\to\birthPart{(-)})$.
	\item The \emph{constant part} functor is defined as $\constantPart{(-)}=\im(\ancientPart{(-)}\to\immortalPart{(-)})$.
\end{enumerate}
\end{defi}

\begin{rem}\label{rem:constanPart}
The universal property of epi-mono-factorizations implies that we have a canonical isomorphism $\constantPart{V}\cong\im(\Delta\lim V\to \Delta\colim V)$ for all $V$.
\end{rem}

We will also form kernels and cokernels of the above composite morphisms. In the cases that we are interested in, these turn out to be the same: Following \cite[Lemma~2.2.4]{MR2977522}, pullbacks of monos and pushouts of epis exist in p-exact categories, and we have canonical isomorphisms
\begin{align*}
\ker(\deathPart{V}\to\birthPart{V})\cong \deathPart{V}&\times_{V}\ancientPart{V}\cong\ker(\ancientPart{V}\to\immortalPart{V})\\
\coker(\deathPart{V}\to\birthPart{V})\cong \birthPart{V}&+_{V}\immortalPart{V}\cong\coker(\ancientPart{V}\to\immortalPart{V})
\end{align*}
for any $V$. Using this fact, we can make the following well-posed definition.

\begin{defi}
We define the following functors $\mathbf{A^T}\to\mathbf{A^T}$.
\begin{enumerate}
	\item The \emph{ancient mortal part} functor is defined as 
	\[\ancientDeathPart{(-)}=\ker(\deathPart{(-)}\to\birthPart{(-)})=\ker(\ancientPart{(-)}\to\immortalPart{(-)}).\]
	\item The \emph{immortal nascent part} functor is defined as 
	\[\immortalBirthPart{(-)}=\coker(\deathPart{(-)}\to\birthPart{(-)})=\coker(\ancientPart{(-)}\to\immortalPart{(-)}).\]
	\end{enumerate}
\end{defi}

We give a common name to all the functors defined above.

\begin{defi}\label{defi:lifespan_diagram}
For an object $V$ in $\mathbf{A^T}$, we will call the diagram %
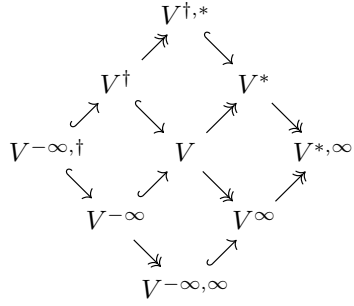
\[
\small
\begin{tikzcd}[row sep={6ex,between origins},column sep={6ex,between origins}]
&&\finitePart{V}\arrow[dr,hook]&&\\
&\deathPart{V} \arrow[ur,two heads]\arrow[rd, hook] & &\birthPart{V}\arrow[dr,two heads] &\\
\ancientDeathPart{V}\arrow[ur,hook]\arrow[dr,hook]&& V \arrow[ru, two heads] \arrow[rd, two heads] &&\immortalBirthPart{V}\\
&\ancientPart{V}\arrow[dr,two heads] \arrow[ru, hook] & &\immortalPart{V}\arrow[ur,two heads]&\\
&&\constantPart{V}\arrow[ur,hook]&&
\end{tikzcd} 
\] 
the \emph{lifespan diagram of $V$}. We call the functors at the nodes of the diagram \emph{lifespan functors} and the natural maps between them \emph{lifespan transformations}.
\end{defi}

Note that the lifespan diagram simplifies to a smaller diagram in many applications. For example, the short exact sequence $\ancientDeathPart{V} \hookrightarrow \ancientPart{V} \twoheadrightarrow \constantPart{V}$ on the bottom left vanishes if $V$ is bounded below.
Similarly, the bottom right sequence vanishes if $V$ is bounded above. For the top left and the top right short exact sequences in the lifespan diagram, we have the following conditions.

\begin{prop}\label{prop:vanishing_parts}
Consider $\lim,\colim\colon\mathbf{A^{T}}\to\mathbf{A}$ and an object $V$ in $\mathbf{A^{T}}$.
\begin{enumerate}
\item If $\colim$ is exact, then $\deathPart{V}=0$ if and only if all structure maps of $V$ are mono.
\item If $\lim$ is exact, then $\birthPart{V}=0$ if and only if all structure maps of $V$ are epi.
\end{enumerate}
\end{prop}
\begin{proof}
We only show the first statement since the second one is dual to it. So, assume that taking colimits is exact.

If $\deathPart{V}=0$, then $V\to\Delta\colim V$ is mono, i.e., $V_{t}\to\colim V$ is mono for any $t\in T$. Now, for any structure map $V_{t}\to V_{u}$, we obtain that the composition $V_{t}\to V_{u}\to \colim V$ is mono since it is equal to the natural map $V_{t}\to\colim V$. This implies that $V_{t}\to V_{u}$ is mono.

Next, assume that all structure maps of $V$ are mono and let $t\in T$. Define an object $\tilde{V}$ in $\mathbf{A^{T}}$ by setting $\tilde{V}_{s}=V_{s}$ for any $s<t$ and $\tilde{V}_{u}=V_{t}$ for any $u\geq t$. There is an obvious map $\tilde{V}\to V$ consisting of structure maps of $V$ and because we assume these structure maps to be mono the map $\tilde{V} \to V$ is mono, too. We assume that taking colimits is exact, so the induced map $\colim\tilde{V}\to\colim V$ is still mono. But $\colim\tilde{V}$ is $V_{t}$ and the induced map is given by the natural map $V_{t}\to\colim V$. Hence, $V\to\Delta\colim V$ is mono, which implies $\deathPart{V}=0$.
\end{proof}

The construction of the lifespan functors involves kernels, cokernels, and images of the natural transformations $\limCounit$ and $\colimUnit$. 
Note, however, that we have not used $\ker\limCounit_{(-)}$ and $\coker\colimUnit_{(-)}$ so far.
These play a somewhat different role than the lifespan functors, as they do not yield subobjects or quotients of the object we start with.
Still, their properties will be of similar importance.

\begin{defi}\label{defi:two_more_functors}
We define the following functors $\mathbf{A^T}\to\mathbf{A^T}$.
\begin{enumerate}
	\item The \emph{ghost complement} functor is defined as $\ghostlikePart{(-)}=\ker\limCounit_{(-)}$.  
	\item The \emph{unborn complement} functor is defined as $\unbornPart{(-)}=\coker\colimUnit_{(-)}$.  
\end{enumerate}
\end{defi}

\subsection{Lifespan Functors and Images, Kernels, and Cokernels}\label{subsec:lifespan_images}
One of our overall goals is to study images, kernels, and cokernels of morphisms in persistent homology. For that purpose, we want to study how the lifespan functors appearing in the statement of \cref{prop:abs_rel_single_filtration} behave with respect to these operations. The relevant theorems hold in the general setting, so, as before, let $\mathbf{A}$ be p-exact with $\mathbf{T}$-indexed limits and colimits.

\begin{ex}
The following examples show that the nascent and mortal part do not preserve images. For both examples, let the index set be $\mathbb{Z}$.
\begin{enumerate}
	\item Consider a morphism $\varphi\colon C([0,+\infty))\to C([0,1])$ which has maximal rank everywhere, e.g., by taking $\varphi_0$ and $\varphi_1$ to be identities and all other maps $0$. Clearly, $\varphi$ is epi and in particular 
	$
	\deathPart{(\im\varphi)}=\deathPart{C([0,1])}=C([0,1]).
	$
	However, we have $\deathPart{C([0,+\infty))}=0$, so $\im\deathPart{\varphi}=0$ and thus $\im\deathPart{\varphi}\neq \deathPart{(\im\varphi)}$.
	\item Now let $\varphi\colon C([-1,0])\to C((-\infty,0])$ be of maximal rank everywhere. By a similar argument, we get $\im\birthPart{\varphi}=0$ but $\birthPart{(\im\varphi)}=C([-1,0])$.
\end{enumerate}
\end{ex}

While the preservation of images fails in general, there are classes of morphisms for which we get the desired result. 
We start with a lemma.

\begin{lem}\label{lem:birth_epi_death_mono}
Let $V$ and $W$ be objects in $\mathbf{A^{T}}$ and $\varphi\colon V\to W$ a morphism.
\begin{enumerate}
\item 
	If $\varphi$ is epi, then $\birthPart{\varphi}$ is epi;
	if $\varphi$ is mono, then $\deathPart{\varphi}$ is mono.
\item
	If $\lim\varphi$ is epi, then $\ancientPart{\varphi}$ is epi;
	if $\colim\varphi$ is mono, then $\immortalPart{\varphi}$ is mono.
\end{enumerate}
\end{lem}
\begin{proof}
Assume $\varphi$ is epi. Note that the canonical map $P\twoheadrightarrow\birthPart{P}$ is also epi for any $P$. We get a commutative diagram
\[
\begin{tikzcd}
V\arrow[r,"\varphi", two heads]\arrow[d, two heads] & W\arrow[d, two heads]\\
\birthPart{V}\arrow[r,"\birthPart{\varphi}"] & \birthPart{W}
\end{tikzcd}
\]
where the composition $V\to\birthPart{W}$ is epi. Thus, $\birthPart{\varphi}$ must be epi, too. The other assertions can be shown analogously.
\end{proof}

\begin{thm}\label{thm:image_parts}
Let $V$ and $W$ be objects in $\mathbf{A^{T}}$ and $\varphi\colon V\to W$ a morphism.
\begin{enumerate}

\item If $\ker\colim\varphi=0$, we have canonical isomorphisms
\begin{align*}
\ker\deathPart{\varphi}&\cong\ker\varphi , &
\ker\immortalPart{\varphi}&=0 , \\
\im\deathPart{\varphi}&\cong\deathPart{(\im\varphi)} , &
\im\immortalPart{\varphi}&\cong\immortalPart{(\im\varphi)}\cong\immortalPart{V} .
\end{align*}

\item If $\coker\lim\varphi=0$, we have canonical isomorphisms
\begin{align*}
\coker\birthPart{\varphi}&\cong\coker\varphi , &
\coker\ancientPart{\varphi}&= 0 , \\
\im\birthPart{\varphi}&\cong\birthPart{(\im\varphi)} , &
\im\ancientPart{\varphi}&\cong\ancientPart{(\im\varphi)}\cong\ancientPart{W}.
\end{align*}
\end{enumerate}
\end{thm} 
\begin{proof}
We only show the first part of the theorem, the second one being completely dual.
First, assume that $\ker\colim\varphi=0$, i.e., $\colim\varphi$ is mono. Taking kernels is left exact, so we have an exact sequence
\[
\begin{tikzcd}
0\arrow[r] & \ker\deathPart{\varphi}\arrow[r] & \ker\varphi\arrow[r] & \ker\immortalPart{\varphi}
\end{tikzcd}
\]
induced by the corresponding sequences from the lifespan diagrams of $V$ and $W$.
By the second part of \cref{lem:birth_epi_death_mono}, our assumption that $\colim\varphi$ is mono implies that $\immortalPart{\varphi}$ is mono. This implies $\ker\immortalPart{\varphi}=0$, and by exactness of the above sequence also $\ker\deathPart{\varphi}\cong\ker\varphi$. In addition, we obtain $\im\immortalPart{\varphi}\cong\immortalPart{V}$ because $\immortalPart{\varphi}$ is mono.

For the assertion on images, consider the epi-mono-factorizations
\[
\begin{tikzcd}
\deathPart{V}\arrow[r, two heads] & \im\deathPart{\varphi}\arrow[r,hook] & \deathPart{W}
\end{tikzcd}
\qquad
\text{and}
\qquad
\begin{tikzcd}
V\arrow[r,"p", two heads] & \im\varphi\arrow[r,"i",hook] & W
\end{tikzcd}
\]
of $\varphi$ and $\deathPart{\varphi}$. Applying the mortal part functor to the second factorization and leaving the first one as is yields a commutative diagram
\[
\begin{tikzcd}
\deathPart{V}\arrow[r, two heads]\arrow[dr, "\deathPart{p}", swap] & \im\deathPart{\varphi}\arrow[r, hook] & \deathPart{W}\\
&\deathPart{(\im\varphi)}\arrow[ur, "\deathPart{i}", swap]&
\end{tikzcd}
\]
By the universal property of epi-mono-factorizations, we get $\im\deathPart{\varphi}\cong\deathPart{(\im\varphi)}$ if $\deathPart{i}$ is mono and $\deathPart{p}$ is epi. Since $i$ is mono, by \cref{lem:birth_epi_death_mono} $\deathPart{i}$ is mono, too. 
By assumption, $\colim\varphi=\colim i\circ\colim p$ is also mono, so $\colim p$ is mono. Using the second part of \cref{lem:birth_epi_death_mono}, we get that $\immortalPart{p}$ is mono as well.
Thus, applying the snake lemma (which holds in p-exact categories, see \cite[Lemma 6.2.8]{MR2977522}) to the diagram
\[
\begin{tikzcd}
0\arrow[r] & \deathPart{V}\arrow[r]\arrow[d,"\deathPart{p}"] & V\arrow[r,]\arrow[d,two heads,"p"] & \immortalPart{V} \arrow[d,hook,"\immortalPart{p}"]\arrow[r] & 0 \\
0\arrow[r] & \deathPart{(\im\varphi)}\arrow[r] & \im\varphi\arrow[r]& \immortalPart{(\im\varphi)}\arrow[r] & 0
\end{tikzcd}
\]
yields that $\deathPart{p}$ is epi. Hence, we obtain $\im\deathPart{\varphi}\cong\deathPart{(\im\varphi)}$ as claimed.

Moreover, recall that $\colim$ preserves epis, so $\colim p$ is not only mono but in fact an isomorphism. Thus, we get a commutative diagram
\[
\begin{tikzcd}
V\arrow[r,two heads]\arrow[d,"p",two heads] & \im\colimUnit_V\arrow[r,hook] & \Delta\colim(V)\\
\im\varphi\arrow[r,two heads] & \im\colimUnit_{\im\varphi}\arrow[r,hook] & \Delta\colim(\im\varphi)\arrow[u,"(\Delta\colim p)^{-1}",hook,swap]
\end{tikzcd}
\]
with the epi-mono-factorizations of $\colimUnit_{V}$ and $\colimUnit_{\im\varphi}$ in the rows. Uniqueness of the epi-mono-factorization implies that the middle terms have to agree, so we obtain
$\immortalPart{V}=\im\colimUnit_V\cong\im\colimUnit_{\im\varphi}=\immortalPart{(\im\varphi)}$. We have already observed that $\im\immortalPart{\varphi}\cong \immortalPart{V}$, so we obtain $\im\immortalPart{\varphi}\cong\immortalPart{(\im\varphi)}$.
\end{proof}

\section{Lifespan of Persistence Modules}\label{sec:lifespan_persistence}
We will look at the special case of the lifespan functors for persistence modules and describe their effect at the level of barcodes in \cref{subsec:lifespan_barcodes}. We then discuss how lifespan functors behave under dualization of persistence modules in \cref{subsec:lifespan_dual}.

\subsection{Lifespan Functors and Barcodes}\label{subsec:lifespan_barcodes}
In order to give an explicit description of how our lifespan functors change the barcode of an interval-decomposable persistence module we will take a detour via matching diagrams. We can apply the theory of lifespan functors to them because they have limits and colimits, as we will show using the component set from \cref{defi:component_set}.

\begin{prop}\label{prop:mch_lim_colim}
Every matching diagram $D$ indexed by $\mathbf{T}$ has a limit and a colimit.
\end{prop}
\begin{proof}
The limit is
given by
\[
\lim D=\{Q\in\mathcal{C}(D)\mid\supp(Q)~\text{is not strictly bounded below}\},
\]
with natural maps $\lim D\to D_{t}$ matching a class $Q$ to its representative in $D_{t}$ if there is one. %
We can also explicitly construct the colimit of $D$ as
\[
\colim D=\{Q\in\mathcal{C}(D)\mid\supp(Q)~\text{is not strictly bounded above} \}.
\]
Here, the natural maps $D_{t}\to\colim D$ match an element to its equivalence class if this class is contained in the set above. 
We omit the straightforward verification that these construction satisfy the universal properties of limits and colimits.
\end{proof}

\begin{rem}
The construction above can be adapted to show that $\mathbf{Mch}$ not only has totally ordered limits and colimits, but all cofiltered limits and filtered colimits.
\end{rem}

We will now look at how the lifespan functors behave when being transported to barcodes via the equivalence in \cref{thm:equivalence}. We introduce some notation.

\begin{defi}\label{defi:bounded_intervals}
We define the following subsets of the intervals $\mathfrak{I}$ in $T$.
\begin{align*}
\birthPart{\mathfrak{I}} &= \{I\in \mathfrak{I}\mid I~\text{is strictly bounded below} \}
,&
\ancientPart{\mathfrak{I}} &= \mathfrak{I}\setminus\birthPart{\mathfrak{I}}
,\\
\deathPart{\mathfrak{I}} &= \{I\in \mathfrak{I}\mid I ~\text{is strictly bounded above} \}
,&
\immortalPart{\mathfrak{I}} &= \mathfrak{I}\setminus\deathPart{\mathfrak{I}}
,\\
\finitePart{\mathfrak{I}} &= \birthPart{\mathfrak{I}}\cap\deathPart{\mathfrak{I}}
,&
\constantPart{\mathfrak{I}} &= \ancientPart{\mathfrak{I}}\cap\immortalPart{\mathfrak{I}}
,\\
\ancientDeathPart{\mathfrak{I}} &= \ancientPart{\mathfrak{I}}\cap\deathPart{\mathfrak{I}}
,&
\immortalBirthPart{\mathfrak{I}} &= \birthPart{\mathfrak{I}}\cap\immortalPart{\mathfrak{I}}
.
\end{align*}
If $B$ is a barcode, we also define
\[
\lifespan{B}=\{(I,a)\in B\mid I\in \lifespan{\mathfrak{I}}\}
\]
for any lifespan functor $\lifespan{(-)}$.
\end{defi}

\begin{thm}\label{thm:barcode_parts}
Let $B$ be a barcode. We have
\[
\mathcal{B}(\lifespan{\mathcal{D}(B)})\cong \lifespan{B}
\]
for all lifespan functors $\lifespan{(-)}$. Moreover, under these isomorphisms, all lifespan transformations correspond to the respective inclusions and coinclusions.
\end{thm}
\begin{proof}
From the definitions of $\mathcal{B}$ and $\mathcal{D}$ as well as the explicit constructions of limits and colimits for matching diagrams in the proof of \cref{prop:mch_lim_colim}, we obtain
\begin{align*}
\mathcal{B}(\Delta\lim\mathcal{D}(B)) &\cong \{(T,(I,a))\in\mathfrak{I}\times B \mid I\in\ancientPart{\mathfrak{I}}\}
,
\\
\mathcal{B}(\Delta\colim\mathcal{D}(B)) &\cong \{(T,(I,a))\in\mathfrak{I}\times B \mid I\in\immortalPart{\mathfrak{I}}\}.
\end{align*}
The overlap matching
$\mathcal{B}(\limCounit_{\mathcal{D}(B)}) \colon \mathcal{B}(\Delta\lim\mathcal{D}(B)) \to B$ matches every interval $(T,(I,a))$ to $(I,a)$.
Similarly, $\mathcal{B}(\colimUnit_{\mathcal{D}(B)})$ matches every element $(I,a)$ with $I\in\immortalPart{\mathfrak{I}}$ to $(T,(I,a))$. All lifespan functors are given on the level of barcodes by first forming kernels, cokernels, and images of $\mathcal{B}(\limCounit_{\mathcal{D}(B)})$ and $\mathcal{B}(\colimUnit_{\mathcal{D}(B)})$, and then kernels, cokernels, and images of the resulting composite lifespan transformations. Hence, the claim follows by applying the formulas for kernels, cokernels, and images of overlap matchings from \cref{prop:ker_overlap_matching} several times.
\end{proof}

Next, we want to show that all the lifespan functors are compatible with the matching module functor $\mathcal{F}$. Since $\mathcal{F}$ is exact, a straightforward proof strategy would be to show that $\mathcal{F}$ also commutes with $\lim$ and $\colim$ and then use the fact that all lifespan functors are obtained from $\lim$ and $\colim$ by forming kernels, cokernels, and images. For colimits, this works out.

\begin{lem}\label{lem:F_colim}
The functor $\mathcal{F}\colon\mathbf{Mch}\to\mathbf{Vec}$ commutes with $\mathbf{T}$-indexed colimits.
\end{lem}
\begin{proof}
Recall that in the proof of \cref{prop:mch_lim_colim} we constructed the colimit of a matching diagram $D$ as the set of components $Q\in\mathcal{C}(D)$ whose support is in $\immortalPart{\mathfrak{I}}$. Further, recall from the definition of the component set that each component can be regarded as a matching diagram. As such, $D$ is canonically isomorphic to the disjoint union (which is not the coproduct, but rather a \emph{butterfly product} in $\mathbf{Mch}$, cf. \cite[Section~2.1.7]{MR2977522}) of all its components. Clearly, $\mathcal{F}$ takes disjoint unions to direct sums. Moreover, for each component $Q\in\mathcal{C}(D)$ the colimit of $\mathcal{F}(Q)$ is one-dimensional if $\supp Q\in\immortalPart{\mathfrak{I}}$ and trivial else. Altogether, we obtain a natural isomorphism
\[
\colim\mathcal{F}(D)\cong\colim\bigoplus_{Q\in\mathcal{C}(D)}\mathcal{F}(Q)
\cong\bigoplus_{Q\in\mathcal{C}(D)}\colim \mathcal{F}(Q)\cong\bigoplus_{\substack{Q\in\mathcal{C}(D)\\ \mathclap{\supp(Q)\in\immortalPart{\mathfrak{I}}(T)}}}\mathbb{F}\cong\mathcal{F}(\colim D),
\]
proving the claim.
\end{proof}

In contrast to colimits, $\mathcal{F}$ generally does not commute with $\mathbf{T}$-indexed limits:
Consider the matching diagram $D$ indexed by the negative integers and given by $D_{-n}=\{1,\dots,n\}$ with structure maps matching each number to itself. Then $\mathcal{F}(\lim D)=\bigoplus_{n\in\mathbb{N}}\mathbb{F}$, but $\lim\mathcal{F}(D)=\prod_{n\in\mathbb{N}}\mathbb{F}$.

Instead, we will use a more explicit argument to show that $\mathcal{F}$ commutes with the ancient part, which, together with the colimit, can also be used as a starting point to construct the other lifespan functors by forming kernels, cokernels, and images.

\begin{thm}\label{thm:lifespan_matchingmodule_commute}
Let $D$ be a matching diagram. 
We have canonical isomorphisms
\[\lifespan{\mathcal{F}(D)}\cong\mathcal{F}(\lifespan{D})\] 
for all lifespan functors $\lifespan{(-)}$, which commute with the lifespan transformations.
\end{thm}
\begin{proof}
We start by showing that $\mathcal{F}$ commutes with the ancient part. For this, consider the epi-mono-factorizations
$\Delta\lim D \twoheadrightarrow  \ancientPart{D} \hookrightarrow D$
and
$\Delta\lim\mathcal{F}(D) \twoheadrightarrow \ancientPart{\mathcal{F}(D)} \hookrightarrow \mathcal{F}(D)$.
Recall that $\mathcal{F}$ preserves exactness and hence also monos and epis. Thus, by applying $\mathcal{F}$ to the first diagram, we get another epi-mono-factorization.
The universal property of the limit also induces a unique morphism $\mathcal{F}(\Delta\lim D)\to\Delta\lim\mathcal{F}(D)$ through which the cone morphism $\mathcal{F}(\Delta\lim D)\to\mathcal{F}(D)$ factors. We obtain a commutative diagram
\[
\begin{tikzcd}
\mathcal{F}(\Delta\lim D)\arrow[r, two heads]\arrow[d]& \mathcal{F}(\ancientPart{D})\arrow[r, hook] & \mathcal{F}(D)\\
\Delta\lim\mathcal{F}(D)\arrow[r, two heads] & \ancientPart{\mathcal{F}(D)}\arrow[r, hook] & \mathcal{F}(D)\arrow[u,hook,"\id"]
\end{tikzcd}
\]
Since epi-mono factorizations are unique up to unique isomorphism, we only need to show that the composite morphism $\mathcal{F}(\Delta\lim D)\to\ancientPart{\mathcal{F}(D)}$ is epi in order to obtain our claim. So let $t_{0}\in T$ and $m\in\ancientPart{\mathcal{F}(D)}_{t_{0}}$. Because $m$ is in the ancient part, i.e., the image of the natural map $\lim\mathcal{F}(D)\to\mathcal{F}(D)_{t_{0}}$, there exists a family $(m_{t})_{t}$ with $m_{t}\in\mathcal{F}(D)_{t}$, $m_{t_{0}}=m$ and such that $\mathcal{F}(D)_{s,t}(m_{s})=m_{t}$ whenever $s\leq t$. Now, write finite formal linear combinations $m_{t}=\sum_{\alpha\in A_{t}}\lambda_{\alpha,t}d_{\alpha,t}$ with $d_{\alpha,t}\in D_{t}$ and $\lambda_{\alpha,t}\neq 0$ for all $t\in T$, $\alpha\in A_{t}$. Because $\mathcal{F}(D)_{s,t_{0}}(m_{s})=m_{t_{0}}$ holds for any $s\leq t_{0}$, we obtain that for any $\alpha\in A_{t_{0}}$ and $s\leq t_{0}$ there exists $\beta=\beta(\alpha,s)\in A_{s}$ with $(d_{\beta,s},d_{\alpha,t_{0}})\in D_{s,t_{0}}$. In particular, the component $Q_{\alpha}$ represented by $d_{\alpha,t_{0}}$ has support in $\ancientPart{\mathfrak{I}}$ for any $\alpha\in A_{t_{0}}$. Thus, $m$ is the image of $\sum_{\alpha\in A_{t_{0}}}\lambda_{\alpha,t_{0}}Q_{\alpha}\in\mathcal{F}(\lim D)$ under the composite morphism $\mathcal{F}(\lim D)\to\ancientPart{\mathcal{F}(D)}_{t_{0}}$. In particular, this map is epi as we needed to show.

The claimed isomorphisms for the other lifespan functors can now be deduced from the isomorphisms we have shown already: Consider the commutative squares
\[
\begin{tikzcd}
\mathcal{F}(D)\arrow[d,"\id"]\arrow[r,"\colimUnit_{\mathcal{F}(D)}"] &\Delta\colim\mathcal{F}(D)\arrow[d] \\
\mathcal{F}(D)\arrow[r,"\mathcal{F}(\colimUnit_{D})"] & \mathcal{F}(\Delta\colim D)
\end{tikzcd}
\qquad
\text{and}
\qquad
\begin{tikzcd}
\ancientPart{\mathcal{F}(D)}\arrow[r,"\alpha_{\mathcal{F}(D)}"]\arrow[d] & \mathcal{F}(D)\arrow[d,"\id"]\\
\mathcal{F}(\ancientPart{D})\arrow[r,"\mathcal{F}(\alpha_{D})"] & \mathcal{F}(D)
\end{tikzcd}
\]
We have just shown that the vertical maps in the square on the right are isomorphisms. The vertical maps in the square on the left are isomorphisms because $\mathcal{F}$ and $\colim$ commute by \cref{lem:F_colim}. Thus, we obtain 
\begin{align*}
\deathPart{\mathcal{F}(D)}=\ker\colimUnit_{\mathcal{F}(D)}&\cong\mathcal{F}(\ker\colimUnit_{D})=\mathcal{F}(\deathPart{D})\\
\immortalPart{\mathcal{F}(D)}=\im\colimUnit_{\mathcal{F}(D)}&\cong\mathcal{F}(\im\colimUnit_{D})=\mathcal{F}(\immortalPart{D})\\
\birthPart{\mathcal{F}(D)}=\coker\alpha_{\mathcal{F}(D)}&\cong\mathcal{F}(\coker\alpha_{D})=\mathcal{F}(\birthPart{D}).
\end{align*}
By these isomorphisms, the vertical maps in the commutative squares
\[
\begin{tikzcd}
\deathPart{\mathcal{F}(D)}\arrow[r,"\beta_{\mathcal{F}(D)}"]\arrow[d] & \birthPart{\mathcal{F}(D)}\arrow[d]\\
\mathcal{F}(\deathPart{D})\arrow[r,"\mathcal{F}(\beta_{D})"] & \mathcal{F}(\birthPart{D})
\end{tikzcd}
\qquad
\text{and}
\qquad
\begin{tikzcd}
\ancientPart{\mathcal{F}(D)}\arrow[r,"\gamma_{\mathcal{F}(D)}"]\arrow[d] & \immortalPart{\mathcal{F}(D)}\arrow[d]\\
\mathcal{F}(\ancientPart{D})\arrow[r,"\mathcal{F}(\gamma_{D})"] & \mathcal{F}(\immortalPart{D})
\end{tikzcd}
\]
are isomorphisms, too. This yields
\begin{align*}
\ancientDeathPart{\mathcal{F}(D)}=\ker\beta_{\mathcal{F}(D)}&\cong\mathcal{F}(\ker\beta_{D})=\mathcal{F}(\ancientDeathPart{D})\\
\finitePart{\mathcal{F}(D)}=\im\beta_{\mathcal{F}(D)}&\cong\mathcal{F}(\im\beta_{D})=\mathcal{F}(\finitePart{D})\\
\immortalBirthPart{\mathcal{F}(D)}=\coker\gamma_{\mathcal{F}(D)}&\cong\mathcal{F}(\coker\gamma_{D})=\mathcal{F}(\immortalBirthPart{D})\\
\constantPart{\mathcal{F}(D)}=\im\gamma_{\mathcal{F}(D)}&\cong\mathcal{F}(\im\gamma_{D})=\mathcal{F}(\constantPart{D}).
\qedhere
\end{align*}
\end{proof}

Finally, combining the fact that $\mathcal{F}$ commutes with the lifespan functors by \cref{thm:lifespan_matchingmodule_commute} with the fact that passing from barcodes to persistence modules is compatible with passing from matching diagrams to persistence modules by \cref{prop:passing_to_modules}, we can use the formulas for the effect of lifespan functors on barcodes from \cref{thm:barcode_parts} to describe how the lifespan functors change the barcodes of persistence modules.

\begin{cor}\label{cor:change_in_barcode}
Let $M$ be a persistence module. If $B$ is a barcode of $M$, then
\[
\lifespan{B}=\{(I,a)\in B\mid I\in\lifespan{\mathfrak{I}}\}
\]
is a barcode for $\lifespan{M}$, where $\lifespan{(-)}$ is any lifespan functor. 
\end{cor}

From \cref{cor:change_in_barcode} we obtain that all the short exact sequences in the lifespan diagram of an interval-decomposable persistence module are obtained up to isomorphism by applying $\mathcal{M}$ to a short exact sequence of barcodes of the form
$B' \hookrightarrow B'\sqcup B'' \twoheadrightarrow B''$.
The inclusion and coinclusion into and out of the disjoint union only match bars with identical underlying intervals, so they admit one-sided inverses by \cref{prop:split_overlap}. Applying $\mathcal{M}$ to the sequence of barcodes above preserves these one-sided inverses, so we obtain the following corollary.

\begin{cor}\label{cor:lifespan_split}
All short exact sequences in the lifespan diagram of an interval-decomposable persistence module split.
\end{cor}

For the unborn complement, the expected formula for the effect on barcodes and the compatibility with $\mathcal{F}$ hold as for the lifespan functors.
We summarize the results and omit the analogous proofs.

\begin{prop}\label{prop:unborn_part}
We have
\[
\mathcal{B}(\unbornPart{\mathcal{D}(B)})\cong\unbornPart{B}:=\{(T\setminus I,a)\mid(I,a)\in B, I\neq T\text{ and }I\in\immortalPart{\mathfrak{I}}\}
\]
for any barcode $B$. 
Moreover, the unborn complement commutes with the matching module functor up to natural transformation. In particular, if $M$ is a persistence module with barcode $B$, then $\unbornPart{B}$ is a barcode of $\unbornPart{M}$.
\end{prop}

For the ghost complement, however, not all of the corresponding statements hold in general: It does not commute with the matching module functor and thus does not change the barcode of a persistence module as we would like it to. This is closely related to the fact that $\mathcal{F}$ does not commute with limits as mentioned before \cref{thm:lifespan_matchingmodule_commute}. The problem disappears for classes of persistence modules where limits do commute with $\mathcal{F}$, e.g. those of \emph{finite type}, which are persistence modules with a finite barcode. Similarly, everything works out as desired if the index set has a smallest element $t_{\min}$, because then, we have $\lim\mathcal{F}(D)\cong\mathcal{F}(D_{t_{\min}})\cong\mathcal{F}(\lim D)$. 

\begin{prop}\label{prop:ghostlike_part}
Assume that $(T,\leq)$ has a smallest element. Then we have
\[
\mathcal{B}(\ghostlikePart{\mathcal{D}(B)})\cong\ghostlikePart{B}:=\{(T\setminus I,a)\mid(I,a)\in B, I\neq T\text{ and }I\in\ancientPart{\mathfrak{I}}\}
\]
for any barcode $B$. 
Moreover, the ghostlike complement commutes with the matching module functor up to natural transformation. In particular, if $M$ is a persistence module with barcode $B$, then $\ghostlikePart{B}$ is a barcode of $\ghostlikePart{M}$.
\end{prop}

\begin{rem}
For persistence modules, some of the lifespan functors admit more explicit descriptions. In particular, the mortal part of a persistence module $M=((M_t)_t,(m_{s,t})_{s,t})$ is the submodule given by the subspaces $\deathPart{M}_t = \bigcup_u\ker m_{t,u}\subseteq M_t$.
In this form, the construction has been considered before by H\"oppner and Lenzing in \cite{MR645942}. They describe it as analogous to taking the submodule of all torsion elements of a module over some integral domain. Certain categories of persistence modules can be shown to be equivalent to categories of modules over some ring (cf. \cite{MR2121296,MR3873177,MR3348168}), and under these equivalences, the mortal part indeed corresponds to the torsion submodule. 
Furthermore, the immortal part has also been considered before in applications of barcodes to symplectic geometry. For a recent example, see \cite{MR4289925}.

Note that, on the other hand, while the ancient part is always a submodule of the intersection of images, $\ancientPart{M}_t \subseteq \bigcap_s\im m_{s,t}$, in general the two need not be isomorphic. The persistence module $M_3$ described in \cref{rem:counterex_inj_proj} provides a counterexample.

\end{rem}

\subsection{Lifespan and Dualization}\label{subsec:lifespan_dual}
When passing from homology to cohomology, we will later see that what happens on the level of persistence modules is dualization. When passing from absolute to relative persistent homology, our correspondence result \cref{prop:abs_rel_single_filtration} will involve the lifespan functors. So, in order to get the full picture involving all four persistence modules associated to a diagram of spaces, we now also have to analyze whether dualization is compatible with lifespan functors.

Because we dualize, we will not only consider persistence modules indexed by $(T,{\leq})$, but also ones indexed by $(T,{\geq})$. When interpreting lifespan in terms of barcodes, it is important to note that the index set changes the meaning of the different classes of intervals we consider, i.e., \cref{defi:bounded_intervals} depends on whether we use the usual or the opposite order.
For example, we have $\birthPart{\mathfrak{I}}(T,{\leq}) = \deathPart{\mathfrak{I}}(T,{\geq})$ and $\ancientPart{\mathfrak{I}}(T,{\leq}) = \immortalPart{\mathfrak{I}}(T,{\geq})$.
Thus, one should expect duals of mortal parts to correspond to nascent parts of duals and so on. To avoid confusion, we introduce some notation.%

\begin{defi}
In the context of the indexing category $\mathbf{T^{\mathrm{op}}}$, 
we will write
\begin{align*}
\deathDualPart{(-)} &:= \birthPart{(-)}
,&
\immortalDualPart{(-)} &:= \ancientPart{(-)}
,&
\birthDualPart{(-)} &:= \deathPart{(-)}
,&
\ancientDualPart{(-)} &:= \immortalPart{(-)}
,\\
\finiteDualPart{(-)} &:= \finitePart{(-)}
,&
\constantDualPart{(-)} &:= \constantPart{(-)}
,&
\immortalBirthDualPart{(-)} &:= \ancientDeathPart{(-)}
,&
\ancientDeathDualPart{(-)} &:= \immortalBirthPart{(-)}
,\\
\unbornDualPart{(-)} &:= \ghostlikePart{(-)}
,&
\ghostlikeDualPart{(-)} &:= \unbornPart{(-)}
.
\end{align*}
\end{defi}
The above convention now yields $\dualLifespan{\mathfrak{I}}(T,\geq) = \lifespan{\mathfrak{I}}(T,\leq)$ for any lifespan functor $\lifespan{(-)}$.

\begin{prop}
Let $M$ be a persistence module. We have canonical isomorphisms
\begin{align*}
\dual{(\deathPart{M})}&\cong\deathDualPart{(\dual{M})}, &
\dual{(\immortalPart{M})}&\cong\immortalDualPart{(\dual{M})}, &
\dual{(\unbornPart{M})}&\cong\unbornDualPart{(\dual{M})}
\end{align*}
\end{prop}
\begin{proof}
The functor $\Hom(-,\mathbb{F})$ takes colimits to limits, so we have a canonical isomorphism $\dual{(\Delta\colim M)}\cong\Delta\lim \dual{M}$. Together with the kernel, cokernel, and image descriptions for dual maps from \cref{lem:dual_maps}, this yields the claim.
\end{proof}

The limit functor for persistence modules commonly exhibits less desirable properties than the colimit functor. For example, the limit functor does not preserve exactness and does not commute with the functor $\mathcal{F}$ while the colimit functor does.
A similar phenomenon arises with dualization of persistence modules, preventing the previous proposition from holding for all lifespan functors: In general, we do not have an isomorphism between $\dual{(\Delta\lim M)}$ and $\Delta\colim\dual{M}$,
because the vector spaces $\Hom (\lim M,\mathbb{F})$ and $\colim\dual{M}$ need not be isomorphic.
However, if $(T,\leq)$ has a smallest element $t_{\min}$, then we have $\dual{(\Delta\lim M)}\cong\Delta\Hom(M_{t_{\min}},\mathbb{F})\cong\Delta\colim\dual{M}$. Thus, we get the following.

\begin{prop}\label{prop:dual_lifespan}
Assume that $(T,\leq)$ has a smallest element and let $M$ be a $\mathbf{T}$-indexed persistence module. Then we have canonical isomorphisms $\dual{(\lifespan{M})}\cong\dualLifespan{(\dual{M})}$ for any lifespan functor $\lifespan{(-)}$.
\end{prop}

For later use, we also record the following completely equivalent reformulation in terms of persistence modules indexed by the opposite order.

\begin{prop}\label{prop:dual_lifespan_2}
Assume that $(T,\leq)$ has a largest element and let $M$ be a $\mathbf{T^{\mathrm{op}}}$-indexed persistence module. Then we have canonical isomorphisms $\dual{(\dualLifespan{M})}\cong\lifespan{(\dual{M})}$ for any lifespan functor $\lifespan{(-)}$.
\end{prop}

Note that above, we do not distinguish $\dual{(-)}$ notationally as a functor from $\mathbf{T}$-indexed persistence modules to  $\mathbf{T^{\mathrm{op}}}$-indexed persistence modules or vice versa.

Furthermore, in the p.f.d.\@ case, applying any lifespan functor $\lifespan{(-)}$ to a persistence module $M$ has the same effect on barcodes as the corresponding functor $\dualLifespan{(-)}$ applied to the dual persistence module $\dual{M}$.

\begin{prop}
Let $M$ be a p.f.d.\@ persistence module. Then
$\lifespan{M}$
and $\dualLifespan{(\dual{M})}$ have the same barcodes for any lifespan functor $\lifespan{(-)}$.
\end{prop}
\begin{proof}
By \cref{lem:dual_barcode} we know that p.f.d.\@ persistence modules have the same barcode as their duals, so the claim follows immediately from the explicit formula in \cref{cor:change_in_barcode} for the effect of lifespan functors on barcodes.
\end{proof}

\section{Projectivity, Injectivity, and Lifespan}
\label{subsec:proj_inj}
As a first application, we will use our lifespan functors to characterize projective and injective objects in the categories of barcodes, matching diagrams, and p.f.d.\@ persistence modules.

\begin{thm}\label{prop:inj_proj_barc}
A barcode $B$ is projective if and only if $\deathPart{B}=0$, and injective if and only if $\birthPart{B}=0$.
\end{thm}
\begin{proof}
We will only show the first statement, the second one can be shown analogously. 
First, assume that $\deathPart{B}=0$. In order to show that $B$ is projective, we consider some overlap matching $\sigma\colon B\to B'$ and need to show that it factors through an arbitrary epi $\tau\colon B''\to B'$. Consider $\sigma$ and $\tau$ as ordinary matchings and set $\rho=\tau'\circ\sigma$, where $\tau'$ is the opposite matching of $\tau$ (see \cref{defi:matchings}). We show that $\rho$ is in fact an overlap matching, i.e., that for any $((I,a),(I'',a''))\in\rho$ we have that $I$ overlaps $I''$ above:

Since $\deathPart{B}=0$, we have $I\in\immortalPart{\mathfrak{I}}$, so that $I$ bounds any other interval, and in particular $I''$, above. What is left to check is that $I''$ bounds $I$ below and that the two intervals have non-empty intersection. If $((I,a),(I'',a''))\in\rho$, then by definition of $\rho$ there is some $(I',a')\in B'$ such that $((I,a),(I',a'))\in\sigma$ and $((I'',a''),(I',a'))\in\tau$. Since $\tau$ is epi, its cokernel vanishes and we obtain $I'\subseteq I''$ from the explicit cokernel formula in \cref{prop:ker_overlap_matching}. Moreover, $I$ overlaps $I'$ above, so we know that $I'$ bounds $I$ below and that $I\cap I'\neq\emptyset$. Together with $I'\subseteq I''$, this implies that $I''$ bounds $I$ below and that $I\cap I''\neq\emptyset$. In total, $I$ overlaps $I''$ above and $\rho$ is an overlap matching. 

Now, an easy calculation verifies that we have $\tau\bullet\rho=\sigma$, i.e., that when considering $\rho$ as an overlap matching, its overlap composition with $\tau$ recovers $\sigma$. Hence, we have shown that $\sigma$ factors through $\tau$, so $B$ is projective.

Next, assume that $\deathPart{B}\neq 0$. We want to show that in this case $B$ is not projective by constructing a barcode $B'$ and an epi $\sigma\colon B'\to B$ that does not split. To do so, choose $(I,a)\in B$ such that $I\in\deathPart{\mathfrak{I}}$, which is possible by our assumption $\deathPart{B}\neq 0$. Define 
\[
J=\{t\in T\mid\text{there exists }s\in I\text{ with }s\leq t\}.
\] 
Clearly, $J$ is an interval in $T$ and it overlaps $I$ above. We define 
\[
B'=(B\setminus\{(I,a)\})\cup\{(J,a)\}
\]
and $\sigma\colon B'\to B$ by matching each element of $B\setminus\{(I,a)\}$ to itself and matching $(J,a)$ to $(I,a)$. This matching $\sigma$ has trivial cokernel since $I\subseteq J$, so $\sigma$ is epi as desired. But, we have $I\neq J$ since $I\in\deathPart{\mathfrak{I}}$. Thus, $\sigma$ matches non-identical intervals and consequently does not split by \cref{prop:split_overlap}, so $B$ cannot be projective.
\end{proof}

Translating via the equivalence of barcodes and matching diagrams, we also obtain that a matching diagram is projective if and only if its mortal part vanishes and injective if and only if nascent part vanishes. 

By \cref{prop:vanishing_parts} we know that vanishing mortal and nascent part can equivalently be described in terms of the structure maps of a diagram, given that taking limits and colimits of diagrams is exact. It is therefore interesting to check whether taking limits and colimits of matching diagrams is exact.

\begin{prop}\label{prop:exact_lim_colim}
The functors $\colim,\lim\colon\mathbf{Mch^{T}}\to\mathbf{Mch}$ are exact.
\end{prop}
\begin{proof}
Let $D \to D' \to D''$
be an exact sequence of matching diagrams. By \cref{prop:F_creates_exactness} the functor $\mathcal{F}$ preserves exactness, so the sequence $\mathcal{F}(D) \to \mathcal{F}(D') \to \mathcal{F}(D'')$
remains exact. It is well-known that taking colimits of persistence modules is exact. Thus, the sequence $\colim\mathcal{F}(D) \to \colim\mathcal{F}(D') \to \colim\mathcal{F}(D'')$
is still exact. Using that $\mathcal{F}$ commutes with taking $\mathbf{T}$-indexed colimits by \cref{lem:F_colim}, we get that $\mathcal{F}(\colim D) \to \mathcal{F}(\colim D') \to \mathcal{F}(\colim D'')$
is also exact. By \cref{prop:F_creates_exactness} the functor $\mathcal{F}$ reflects exactness, so
$\colim D \to \colim D' \to \colim D''$
is exact, proving that taking colimits of matching diagrams is exact. Self-duality of the category $\mathbf{Mch}$ implies that taking limits then has to be exact, too.
\end{proof}

Knowing that taking limits and colimits of matching diagrams is exact, we can now combine the equivalent conditions for vanishing mortal and nascent parts from \cref{prop:vanishing_parts,prop:inj_proj_barc} to obtain the following.

\begin{cor}\label{cor:inj_proj_mchdgm}
A matching diagram $D$ is projective if and only if all of its structure maps are mono, and injective if and only if all of its structure maps are epi.
\end{cor}

A natural question to ask is whether statements analogous to the above \cref{prop:inj_proj_barc,cor:inj_proj_mchdgm} also hold for persistence modules instead of matching diagrams: can we characterize projectivity/injectivity or structure maps being mono/epi by vanishing mortal/nascent parts? We start with some general results.

\begin{prop}\label{prop:death_part_monos}
For any persistence module $M$, we have $\deathPart{M}=0$ if and only if all structure maps of $M$ are mono. Moreover, if $\birthPart{M}=0$, then all structure maps of $M$ are epi.
\end{prop}
\begin{proof}
The first part of the proposition is just a special case of the first part of \cref{prop:vanishing_parts}, noting that taking colimits of persistence modules is exact.

For the second part, we repeat parts of the dual version of the proof of \cref{prop:vanishing_parts}:
If $\birthPart{M}=0$, then $\Delta\lim M\to M$ is epi, i.e., $\lim M\to M_{t}$ is epi for all $t\in T$. This implies in particular that for any structure map $M_{s}\to\ M_{t}$, the composition $\lim M\to M_{s}\to M_{t}$ is epi since it is equal to $\lim M\to M_{t}$. As a consequence, $M_{s}\to M_{t}$ needs to be epi, finishing the proof.
\end{proof}

In the category $\mathbf{vec^{T}}$ of p.f.d.\@  persistence modules, we can indeed characterize projectives and injectives in a way analogous to matching diagrams $\mathbf{Mch^{T}}$.

\begin{thm}\label{thm:proj_inj}
Let $M$ be a p.f.d.\@ persistence module.
\begin{enumerate}
\item The following are equivalent: %
\begin{enumerate}
\item All structure maps of $M$ are mono.
\item $\deathPart{M}=0$.
\item $M$ is projective in $\mathbf{vec^{T}}$.
\end{enumerate}
\goodbreak
\item The following are equivalent: %
\begin{enumerate}
\item All structure maps of $M$ are epi.
\item $\birthPart{M}=0$.
\item $M$ is injective in $\mathbf{vec^{T}}$.
\end{enumerate}
\end{enumerate}
\end{thm}
\begin{proof}
Starting with the first part of the theorem, we first note that we have already shown that $\deathPart{M}=0$ is equivalent to $M$ having mono structure maps for any persistence module $M$ in \cref{prop:death_part_monos}. Thus, what is left to show for the first part is that $\deathPart{M}=0$ is equivalent to $M$ being projective in the p.f.d.\@ category. To do so, we fix a barcode decomposition $M\cong\bigoplus_{\alpha}C(I_{\alpha})$, which is possible by Crawley-Boevey's Theorem since $M$ is p.f.d.

Now, assume that $\deathPart{M}=0$, or equivalently $M=\immortalPart{M}$. We want to show that $M$ is projective in $\mathbf{vec^{T}}$. A direct sum of projectives is projective, so it suffices to check that the interval modules $C(I_{\alpha})$ in the decomposition of $M$ are projective in $\mathbf{vec^{T}}$. Recall that $\mathbf{vec^{T}}$ is abelian, so in order to show that $C(I_{\alpha})$ is projective in $\mathbf{vec^{T}}$ we only need to show now that any epimorphism $\varphi\colon N\to C(I_{\alpha})$ with $N$ p.f.d.\@ splits:

Choosing a barcode decomposition $N\cong\bigoplus_{\beta}C(J_{\beta})$ induces maps $\varphi_{\beta}\colon C(J_{\beta})\to C(I_{\alpha})$ for each $\beta$. Because $\varphi$ is epi and $N$ is p.f.d.\@ there has to be some $\beta_{0}$ such that $\varphi_{\beta_{0}}$ is epi. %
This implies that $I_{\alpha}\subseteq J_{\beta_{0}}$ and that simultaneously $J_{\beta_{0}}$ has to overlap $I_{\alpha}$ above. Since we assume $\deathPart{M}=0$, $I_{\alpha}\in\immortalPart{\mathfrak{I}}$ holds by \cref{cor:change_in_barcode}, so we obtain $I_{\alpha}=J_{\beta_{0}}$, which yields that $\varphi_{\beta_{0}}$ is an isomorphism. We can thus define $\psi\colon C(I)\to N$ as the composition
\[
\begin{tikzcd}
C(I_{\alpha}) \arrow[r,"\varphi_{\beta_{0}}^{-1}"] & C(J_{\beta_{0}}) \arrow[r,hook] & \bigoplus_{\beta}C(J_{\beta}) \cong N.
\end{tikzcd}
\]
By construction, we have $\varphi\circ\psi=\varphi_{\beta_{0}}^{-1}\circ\varphi_{\beta_{0}}$, which is the identity on $C(I_{\alpha})$, so $\varphi$ splits. Thus, we have shown that $C(I_{\alpha})$, and consequently $M$, is projective.

Next, we assume that $\deathPart{M}\neq 0$ and show that $M$ is not projective in $\mathbf{vec^{T}}$. Because the mortal part of $M$ does not vanish, there now has to be some $\alpha_{0}$ with $I_{\alpha_{0}}\in\deathPart{\mathfrak{I}}$. We proceed as in the proof of \cref{prop:inj_proj_barc} and define 
\[
J=\{t\in T\mid\text{there exists }s\in I_{\alpha_{0}}\text{ with }s\leq t\}.
\] 
Clearly, $J$ is an interval in $T$ and it overlaps $I_{\alpha_{0}}$ above. The canonical map $C(J)\to C(I_{\alpha_{0}})$ is an epi, which we can use to obtain an epi 
\[
\bigoplus_{\alpha\neq\alpha_{0}}C(I_{\alpha})\oplus C(J)\to \bigoplus_{\alpha\neq\alpha_{0}}C(I_{\alpha})\oplus C(I_{\alpha_{0}}) \cong M
\]
in $\mathbf{vec^{T}}$, which is an isomorphism on all summands except for $C(J)$. If this map would split, the splitting would induce a morphism $C(I_{\alpha_{0}})\to C(J)$, which cannot exist since $I_{\alpha_{0}}$ by construction does not overlap $J$ above. Thus, the epi we constructed does not split and $M$ is not projective in $\mathbf{vec^{T}}$. This finishes the proof of the first part.

For the second part, we omit showing that $\birthPart{M}=0$ is equivalent to $M$ being injective in $\mathbf{vec^{T}}$ because the proof is very similar to the previous arguments. One thing to note is that for injectivity, one wants to use that products of injectives are again injectives. This can still be done with barcode decompositions in the p.f.d.\@ setting because they can not only be interpreted as decompositions in terms of direct sums, but actually as decompositions in terms of biproducts.

That $\birthPart{M}=0$ implies $M$ having epi structure maps has been shown for all persistence modules $M$ before in \cref{prop:death_part_monos}, so what remains to be checked is that $\birthPart{M}=0$ if $M$ is p.f.d.\@ and its structure maps are epi. To see that this is the case, one can use the fact that the functor $\lim\colon\mathbf{vec^{T}}\to\mathbf{Vec}$ is exact (because derived inverse limits of p.f.d.\@ persistence modules vanish \cite[Proposition 1.1]{MR260839}, \cite[Th\'eor\`eme 2]{MR136640}) and reuse the argument in the proof of \cref{prop:vanishing_parts} to show that $\Delta\lim M\to M$ is epi if the structure maps of $M$ are epi, which implies that $\birthPart{M}=0$.
\end{proof}

Any p.f.d.\@ persistence module has a barcode and that the lifespan functors are compatible with the passage to barcodes, so another way of phrasing the previous theorem is that a p.f.d.\@ persistence module is projective or injective in $\mathbf{vec^{T}}$ if and only if its barcode has the corresponding property in $\mathbf{Barc(T)}$.

\begin{rem}\label{rem:counterex_inj_proj}
When considering persistence modules beyond the p.f.d.\@ category, some of the equivalences established in \cref{thm:proj_inj} do not hold anymore in general.
We give a few examples.

Any projective object in $\mathbf{Vec^{T}}$ has vanishing mortal part and mono structure maps. However, the real-indexed interval module $M_{1}=C(0,\infty)$ satisfies $\deathPart{M_{1}}=0$ and has mono structure maps, but it is not projective in $\mathbf{Vec^{T}}$ because the obvious epi \(\bigoplus_{n\in\mathbb{N}_{>0}} C\left(\frac{1}{n},\infty\right)\to C(0,\infty)=M_{1}\) does not split. For a classification of projectives in $\mathbf{Vec^{T}}$, see \cite{MR596150}.

Similarly, any injective object in $\mathbf{Vec^{T}}$ has epi structure maps, but we are presently unable to determine whether injective persistence modules also need to have vanishing nascent part. In any case, 
the real-indexed persistence module $M_{2}=C(-\infty,0)$ satisfies $\birthPart{M_{2}}$ and has epi structure maps, but it is not injective in $\mathbf{Vec^{T}}$ because the obvious mono \(M_{2}=C(-\infty,0)\to\prod_{n\in\mathbb{N}_{>0}}C\left(-\infty,-\frac{1}{n}\right)\) does not split. For classification results for injectives in $\mathbf{Vec^{T}}$, see \cite{MR645942,MR709843}.

Having epi structure maps also generally does not imply vanishing nascent part for persistence modules: There is a non-zero persistence module $M_{3}$ indexed by the opposite poset of the first uncountable ordinal $\omega_{1}$ whose structure maps are all epi, but which satisfies $\lim M_{3}=0$ (\cite[Section 3]{MR61086}), so that $\birthPart{M_{3}}=M_{3}\neq 0$.
\end{rem}

\section{Functorial Dualities in Persistent Homology}\label{sec:dualities}

In \cref{subsec:absrel}, we discuss functorial versions of the duality results by \textcite{MR2854319}. As an application, we present some considerations in \cref{subsec:image_duality} on obtaining images of morphisms in persistent homology from their counterparts in relative cohomology, which is of great relevance for making their algorithmic computation more efficient.

\subsection{Persistent Homology Dualities in Terms of Lifespan Functors}
\label{subsec:absrel}

We will now prove a  generalization of the absolute-relative correspondence \cite[Proposition 2.4]{MR2854319} involving our lifespan functors. In order for this to work nicely, we only consider filtrations that satisfy the following condition.

\begin{defi}
Let $X$ be a $\mathbf{T}$-indexed diagram of topological spaces. We say that $X$ is \emph{colimit proper} if the natural maps $\colim H_{d}(X)\to H_{d}(\colim X)$ and  $H_{d}(\colim X)\to\lim H_{d}(\colim X,X)$ are isomorphisms for all $d$. 
\end{defi}

Note that colimit properness is always satisfied if the diagram $X$ is initially empty and eventually constant. In particular, if the index set has a largest element $t_{\max}$ and a smallest element $t_{\min}$ then every $X$ with $X_{t_{\min}}=\emptyset$ is colimit proper. These properties are usually given in the computational setting for persistent homology.

\begin{thm}\label{prop:abs_rel_single_filtration}
Let $X$ be a colimit proper filtration of topological spaces. For all $d$, we have the following isomorphisms, which are natural in $X$:
 \begin{align*}
 \deathPart{H_{d-1}(X)}&\cong\birthPart{H_d(\colim X,X)},
 \\
 \unbornPart{H_{d}(X)}&\cong\ancientPart{H_{d}(\colim X,X)},
 \\
 \immortalPart{H_{d}(X)}&\cong\ghostlikePart{H_{d}(\colim X,X)}.
\end{align*}
\end{thm}
\begin{proof}
To shorten notation, we write $A$ for $\colim X$. Since $X$ is a filtration, the natural map $C_{*}(X)\to C_{*}(\Delta A)$ is mono. Thus, we have a short exact sequence
\[
\begin{tikzcd}
0\arrow[r] & C_{*}(X)\arrow[r] & C_{*}(\Delta A)\arrow[r] & C_{*}(A, X)\arrow[r] & 0
\end{tikzcd}
\]
This induces a long exact sequence of persistence modules
\[
\begin{tikzcd}
\cdots \arrow[r] & \Delta H_d(A) \arrow[r, "\limCounit_d"] & H_d(A,X) \arrow[r, "\partial"] & H_{d-1}(X) \arrow[r,"\colimUnit_{d-1}"] & \Delta H_{d-1}( A) \arrow[r] & \cdots 
\end{tikzcd}
\]
Since we assume $X$ to be colimit proper, the map $\limCounit_d$ can be identified with the counit 
\[\limCounit_{H_d(A,X)}\colon \Delta\lim H_d(A,X) \to H_d(A,X)\]
of the adjunction $\Delta\dashv\lim$.
Similarly, the map $\colimUnit_{d-1}$ may be identified with the unit
\[\colimUnit_{H_{d-1}(X)}\colon H_{d-1}(X)\to \Delta\colim H_{d-1}(A,X)\]
of the adjunction $\colim\dashv\Delta$. 
Applying the definition of the lifespan functors,
the claimed isomorphisms are now simply given by exactness of the above sequence
:
\begin{align*}
\deathPart{H_{d-1}(X)}\cong\ker\colimUnit_{d-1}&\cong \coker\limCounit_d\cong\birthPart{H_d(A,X)},
\\
\unbornPart{H_{d}(X)}\cong\coker\colimUnit_d&\cong\im\limCounit_d\cong \ancientPart{H_d(A,X)},
\\
\immortalPart{H_d(X)}\cong\im\colimUnit_d&\cong\ker\limCounit_d\cong \ghostlikePart{H_d(A,X)}.
\end{align*}

These isomorphisms are natural in $X$ as a direct consequence of the fact that the construction of the long exact sequence is natural in $X$.
\end{proof}

Using the barcode formulas for lifespan functors and complements in \cref{cor:change_in_barcode,prop:unborn_part,prop:ghostlike_part}, one can easily recover the original duality result by de Silva et al.\@ \cite[Proposition~2.4]{MR2854319}.
Moreover, naturality in the filtration variable implies that for a morphism $f \colon X \to Y$ between colimit proper filtrations with $\phi=\colim f$ we also get isomorphisms 
 \begin{align*}
\deathPart{H_{d-1}(f)}&\cong\birthPart{H_d(\phi,f)},
&
\unbornPart{H_{d}(f)}&\cong\ancientPart{H_{d}(\phi,f)},
&
\immortalPart{H_{d}(f)}&\cong\ghostlikePart{H_{d}(\phi,f)}
\end{align*}
in the category of morphisms of persistence modules. These also translate to isomorphisms between the corresponding images, kernels, and cokernels.

Note that the isomorphism between the mortal part of the absolute persistent homology and the nascent part of the relative persistent homology in the proof of \cref{prop:abs_rel_single_filtration} is induced by the boundary operator. This means that if an interval in the nascent part of the barcode of the relative persistent homology is represented by some relative cycle, the boundary of this cycle represents the same interval in the absolute persistent homology in one dimension lower, as observed in
\cite{MR2854319}.

\begin{rem}
While the above result is stated for persistent homology of filtrations of spaces, a similar statement holds in the purely algebraic setting.
Given a filtered chain complex $C$, we can consider the short exact sequence
\[
\begin{tikzcd}
0\arrow[r] & C\arrow[r] & \Delta\colim C\arrow[r] & \unbornPart{C}\arrow[r] & 0
\end{tikzcd}
\]
We can then continue as in the proof above to get natural isomorphisms
 \begin{align*}
 \deathPart{H_{d-1}(C)}&\cong\birthPart{H_d(\unbornPart{C})}
&
 \unbornPart{H_{d}(C)}&\cong\ancientPart{H_{d}(\unbornPart{C})}
&
 \immortalPart{H_{d}(C)}&\cong\ghostlikePart{H_{d}(\unbornPart{C})}.
\end{align*}

\end{rem}

For completeness, we also record a functorial version of the correspondence between persistent homology and persistent cohomology \cite[Proposition 2.3]{MR2854319}, which follows immediately from the universal coefficient theorem.

\begin{prop}\label{prop:hom_cohom}
Let $X$ be a $\mathbf{T}$-indexed diagram of topological spaces. For all $d$, we have the following isomorphisms, which are natural in $X$:
\begin{align*}
\dual{H_{d}(X)}&\cong H^{d}(X),
\\
\dual{H_{d}(\colim X,X)}&\cong H^{d}(\colim X,X).
\end{align*}
\end{prop}
While the correspondence in \cite[Proposition 2.3]{MR2854319} is stated on the level of barcodes, the natural isomorphism asserted in \cref{prop:hom_cohom} appears in its proof, which essentially combines the previous statement with the fact that p.f.d.\@ persistence modules have the same barcode as their duals (\cref{lem:dual_barcode}).

As in the absolute-relative correspondence, naturality in the variable $X$ yields corresponding isomorphisms in the category of morphisms of persistence modules for maps $f \colon X \to Y$.

\subsection{Absolute Homology Images from Relative Cohomology Images}\label{subsec:image_duality}

As a concrete application, we want to explain how to use our previous results for the efficient computation of barcodes for images of morphisms in persistent homology. Note that similar considerations also apply for kernels and cokernels of such morphisms.

As mentioned above, and as is explained e.g.\@ in \cite{MR4298669}, one of the most efficient ways currently known to compute the barcode of the persistent homology of a filtration of finite simplicial complexes is to actually compute the barcode of the persistent relative cohomology with the so-called clearing optimization, and to then translate this to persistent homology via the two duality results from \textcite{MR2854319}.

Our generalizations of these duality results now allow us to proceed similarly for the image of a map $f \colon X \to Y$. Since we are talking about computational speed-ups, $X$ and $Y$ are assumed to be filtrations of finite simplicial complexes indexed by a totally ordered set $\mathbf{T}$ with a smallest element $t_{\min}$ and a largest element $t_{\max}$. We also assume that $X_{t_{\min}} = Y_{t_{\min}} = \emptyset$, so that both filtrations are colimit proper, and that $\colim H_{d}(f) = \lim H_{d}(\colim f, f) = H_{d}(f_{t_{\max}})$ is an isomorphism.

In order to compute the barcode for $\im H_{d}(f)$, we start with applying the (non-natural) decomposition
\[
\im H_{d}(f) \cong 
\deathPart{(\im H_{d}(f))}
\oplus
\immortalPart{(\im H_{d}(f))}
\]
from \cref{cor:lifespan_split}.
We consider both summands separately, making use of the fact that taking barcodes is compatible with direct sums.

Starting with the first summand, we observe that because $\colim H_{d}(f)  = H_{d}(f_{t_{\max}})$ is an isomorphism, and in particular a monomorphism, we have
\[
\deathPart{(\im H_{d}(f))}
\cong
\im (\deathPart{H_{d}(f)})
\]
using the first part of \cref{thm:image_parts}.
The natural duality \cref{prop:abs_rel_single_filtration}, which we can apply since $X$ and $Y$ are colimit proper, provides an isomorphism
\[
\im (\deathPart{H_{d}(f)})
\cong
\im (\birthPart{H_{d+1}(\colim f,f)}).
\]
An application of the second part of \cref{thm:image_parts} yields the isomorphism
\[
\im (\birthPart{H_{d+1}(\colim f,f)})
\cong
\birthPart{(\im H_{d+1}(\colim f,f))}
\]
using that $\lim H_{d}(\colim f, f) = H_{d}(f_{t_{\max}})$ is epi.
Finally, the duality of homology and cohomology from \cref{prop:hom_cohom} yields an isomorphism
\[
\birthPart{(\im H_{d+1}(\colim f,f))}
\cong
(\birthPart{\dual{(\im H^{d+1}(\colim f,f))})},
\]
where we also make use of the fact that applying dualization twice yields the identity on p.f.d\@ persistence modules. Finally, because our index set has a largest element, \cref{prop:dual_lifespan_2} gives
\[
(\birthPart{\dual{(\im H^{d+1}(\colim f,f))})}
\cong
\dual{(\birthDualPart{(\im H^{d+1}(\colim f,f))})}.
\]
In total, the above implies that 
$
\deathPart{(\im H_{d}(f))}
$
and 
$
\birthDualPart{(\im H^{d+1}(\colim f,f))}
$
have the same barcode by \cref{lem:dual_barcode} because we are in the p.f.d.\@ setting, so we can obtain the mortal part of the absolute homology barcode from the one in relative cohomology.

For the second term in the mortal-immortal decomposition of $\im H_{d}(f)$, we have
\[
\immortalPart{(\im H_{d}(f))}
\cong
\immortalPart{H_{d}(X)}
\]
by \cref{thm:image_parts}.
We proceed again with our natural absolute-relative duality from \cref{prop:abs_rel_single_filtration} to obtain
\[
\immortalPart{H_{d}(X)}
\cong
\ghostlikePart{H_{d}(\colim X, X)}.
\]
Since all modules are p.f.d., passing to cohomology with \cref{prop:hom_cohom} yields
\[
\ghostlikePart{(H_{d}(\colim X, X))}
\cong
\ghostlikePart{(\dual{(H^{d}(\colim X, X))})}.
\]
\cref{prop:dual_lifespan_2} finally yields
\[
\ghostlikePart{(\dual{(H^{d}(\colim X, X))})}
\cong
\dual{(\ghostlikeDualPart{H^{d}(\colim X, X)}))}.
\]
Thus, we can also obtain the immortal part of the absolute homology barcode from the one in relative cohomology.

\section*{Acknowledgements}
This research has been supported by the German Research Foundation (DFG) through the Collaborative Research Center SFB/TRR 109 \emph{Discretization in Geometry and Dynamics}, the Collaborative Research Center SFB/TRR 191 \emph{Symplectic Structures in Geometry, Algebra and Dynamics}, the Cluster of Excellence EXC-2181/1 \emph{STRUCTURES}, and the Research Training Group RTG 2229 \emph{Asymptotic Invariants and Limits of Groups and Spaces}.

\printbibliography

\todos

\end{document}